\documentclass[12pt]{article}
\usepackage{amssymb}
\usepackage{amsmath}
\usepackage{mathrsfs}
\usepackage{graphics}
\usepackage{graphicx}
\usepackage{xcolor}
\usepackage{subfigure}
\usepackage[T1]{fontenc}
\usepackage{latexsym,amssymb,amsmath,amsfonts,amsthm}\usepackage{txfonts}
\usepackage[colorlinks=true,linkcolor=black,citecolor=black,urlcolor=black]{hyperref}
\newcommand{\be}{\begin{eqnarray}}
\newcommand{\ben}{\begin{eqnarray*}}
\newcommand{\en}{\end{eqnarray}}
\newcommand{\enn}{\end{eqnarray*}}

\newtheorem{theorem}{Theorem}[section]
\newtheorem{lemma}{Lemma}[section]
\newtheorem{prp}[theorem]{Proposition}
\newtheorem{thm}[theorem]{Theorem}
\newtheorem{cor}[theorem]{Corollary}
\newtheorem{dfn}{Definition}[section]

\newtheorem{remark}{Remark}

\allowdisplaybreaks[4]
\begin{document}
\renewcommand{\theequation}{\arabic{section}.\arabic{equation}}
\begin{titlepage}
\title{\bf  Central limit theorem and moderate deviation principle for stochastic scalar conservation laws
}
\author{Zhengyan Wu$^{1,2}$\ \ Rangrang Zhang$^{3,}\footnote{Corresponding author.}$\\
%{\small $^1$
%RCSDS, Academy of Mathematics and Systems Science, Chinese Academy of Sciences, Beijing 100190, China.} \\
{\small $^1$ Academy of Mathematics and System Science, Chinese Academy of Science, Beijing 100190, China.}\\
{\small $^2$ Department of Mathematics, University of Bielefeld, D-33615 Bielefeld, Germany.}\\
{\small $^3$ School of Mathematics and Statistics,
Beijing Institute of Technology, Beijing, 100081, China.}\\
({\sf zwu@math.uni-bielefeld.de}, {\sf rrzhang@amss.ac.cn}  )}
%\author{ Rangrang Zhang\\
%{\small  School of  Mathematics and Statistics,
%Beijing Institute of Technology, Beijing, 100081, China.}\\
%({\sf rrzhang@amss.ac.cn})}
\date{}
\end{titlepage}
\maketitle

\noindent\textbf{Abstract}:
We establish a central limit theorem and prove a moderate deviation principle for stochastic scalar conservation laws. Due to the lack of viscous term, this is done in the framework of kinetic solution. The weak convergence method and doubling variables method play a key role.

\noindent \textbf{AMS Subject Classification}:\ \ 60F10, 60H15, 60G40

\noindent\textbf{Keywords}: stochastic scalar conservation laws; {\color{black}weak convergence method; doubling variables method;} central limit theorem; moderate deviation principle; kinetic solution.

\section{Introduction}
This paper concerns the asymptotic behaviour of stochastic scalar conservation laws with small multiplicative noise. More precisely, fix any
$T>0$ and let $(\Omega,\mathcal{F},P,\{\mathcal{F}_t\}_{t\in
[0,T]},(\{\beta_k(t)\}_{t\in[0,T]})_{k\in\mathbb{N}})$ be a stochastic basis. Without loss of generality, here the filtration $\{\mathcal{F}_t\}_{t\in [0,T]}$ is assumed to be complete and $\{\beta_k(t)\}_{t\in[0,T]},k\in\mathbb{N}$, are independent (one-dimensional)  $\{\mathcal{F}_t\}_{t\in [0,T]}-$Wiener processes. We use $E$ to denote the expectation with respect to $P$.
Let $\mathbb{T}^d\subset\mathbb{R}^d$ denote the one-dimensional torus (suppose the periodic length is $1$).
We are concerned with the following  conservation laws with stochastic forcing
\begin{eqnarray}\label{sbe}
\left\{
  \begin{array}{ll}
  du+{\rm{ {\rm{div}} }}(A(u))dt=\Phi(u)dW(t)\ \ {\rm{in}}\ \ \mathbb{T}^d\times[0,T],\\
u(\cdot,0)=C \quad {\rm{on}}\ \ \mathbb{T}^d,
  \end{array}
\right.
\end{eqnarray}
where $u:(\omega,x,t)\in\Omega\times\mathbb{T}^d\times[0,T]\mapsto u(\omega,x,t):=u(x,t)\in\mathbb{R}$,
that is, the equation is periodic in the space variable $x\in \mathbb{T}^d$, the flux function $A:\mathbb{R}\to\mathbb{R}^d$ belongs to $C^2(\mathbb{R};\mathbb{R}^d)$, the coefficient
 $\Phi:\mathbb{R}\to\mathbb{R}$ is measurable and fulfills certain conditions specified later, $W$ is a cylindrical Wiener process defined on a given (separable) Hilbert space $U$ with
the form $W(t)=\sum_{k\geq 1}\beta_k(t) e_k,t\in[0,T]$, where $(e_k)_{k\geq 1}$ is a complete orthonormal {\color{black}basis} in the Hilbert space $U$, the initial value $u_0$ is equal to a constant $C$, for simplicity, we take $C=1$. The reason for requiring constant initial values is postponed to explain in the last paragraph of this section.

 \vskip 0.3cm
The (deterministic) conservation laws (in both scalar and vectorial) are fundamental to our understanding of the space-time evolution laws of interesting physical quantities, in that they describe (dynamical) processes that can or cannot occur in nature. When $d=1$ and the flux function $A(\xi)=\frac{\xi^2}{2}$, (\ref{sbe}) is blackuced to  stochastic Burgers equation.
The deterministic Burgers equation was introduced in \cite{B} to describe the turbulence phenomena in fluids, which can be solved by Cole-Hopf transform.
The randomly forced Burgers equation is a prototype for range of problems in non-equilibrium statistical physics where strong effects are included, see \cite{BFKL,BG,BFGLV,CY,F,GM,KS,Pol}, etc.
%On the other hand,
%from the perspective of stochastic conservation laws on $d-$dimensional torus $\mathbb{T}^d$
%\begin{eqnarray}\label{rrr-13}
%\left\{
%  \begin{array}{ll}
%  du+{\rm{ {\rm{div}} }}(A(u))dt=\Phi(u)dW(t)\ \ {\rm{in}}\ \ [0,T]\times\mathbb{T}^d,\\
%u(\cdot,0)=1 \quad {\rm{on}}\ \ \mathbb{T}^d,
%  \end{array}
%\right.
%\end{eqnarray}
% the Burgers equation (\ref{sbe}) is a special example with $d=1$ and the flux function $A(\xi)=\frac{\xi^2}{2}$.
Regarding to the conservation laws (\ref{sbe}), both the deterministic ($\Phi=0$) and stochastic cases have been studied extensively by many people.
For more background on the conservation laws, we refer the readers to the monograph \cite{Dafermos}, the work of Ammar,
Wittbold and Carrillo \cite{K-P-J} and references therein. As we know, the Cauchy problem
for the deterministic conservation laws does not admit any (global) smooth solutions, but there exist infinitely many weak solutions to the deterministic Cauchy problem. To solve the problem of non-uniqueness, an additional entropy condition was added to identify a physical weak solution. Under this condition, Kru\v{z}kov \cite{Kr-1,Kr-2} introduced the notion of entropy solutions for the deterministic first-order scalar conservation laws.
The kinetic formulation of weak entropy solution of the Cauchy problem for a general multi-dimensional scalar conservation laws (also called the kinetic system), was derived by Lions, Perthame and Tadmor in \cite{L-P-T}.
In recent years, the stochastic conservation laws has been developed rapidly. We refer the reader to  \cite{K,V-W,F-N,DWZZ}, etc. We
  particularly mention the paper \cite{D-V-1} in which the authors  proved the existence and uniqueness of kinetic solution to the Cauchy problem for (\ref{sbe}) in any dimension.
 In addition, there are some works on the long time behavior/ergodicity of stochastic scalar conservation laws. {\color{black}In the one-dimensional case}, {\color{black} E, Khanin, Mazel and Sinai} \cite{E00} proved the existence and uniqueness of invariant measures for the periodic stochastic  Burgers equation with additive forcing. Later, Debussche and Vovelle \cite{DV-2} studied scalar conservation laws with additive stochastic forcing on torus of any dimension and proved the existence and uniqueness of an invariant measure for sub-cubic fluxes and sub-quadratic fluxes, respectively.
Recently, for the small noise asymptotic behaviour, Dong et al. \cite{DWZZ} established Freidlin-Wentzell's type large deviation principles (LDP) for the kinetic solution to the scalar stochastic conservation laws.

\vskip 0.3cm

The purpose of this paper is to investigate the central limit theorem (CLT) and moderate deviation principle (MDP) for (\ref{sbe}) driven by small multiplicative noise. Concretely, we consider {\color{black}for $\varepsilon>0$ the perturbed stochastic conservation laws}
\begin{eqnarray}\label{sbe-1}
\left\{
  \begin{array}{ll}
  du^{\varepsilon}+ {\rm{div}} (A(u^{\varepsilon}))dt=\sqrt{\varepsilon}\Phi(u^{\varepsilon})dW(t)\ \ {\rm{in}}\ \ \mathbb{T}^d\times[0,T],\\
u^{\varepsilon}(\cdot,0)=1 \quad {\rm{on}}\ \ \mathbb{T}^d.
  \end{array}
\right.
\end{eqnarray}
We aim to explore the deviations of $u^{\varepsilon}$ from the deterministic solution, as
 $\varepsilon\rightarrow0$.
The deterministic solution satisfies
 \begin{eqnarray}\label{r-12}
\left\{
  \begin{array}{ll}
  d\bar{u}+ {\rm{div}} (A(\bar{u}))dt=0\ \ {\rm{in}}\ \ [0,T]\times\mathbb{T}^d,\\
\bar{u}(\cdot,0)=1 \quad {\rm{on}}\ \ \mathbb{T}^d.
  \end{array}
\right.
\end{eqnarray}
We are interested in the asymptotic behavior of the trajectories,
\begin{equation*}
X^{\varepsilon}(t)=\frac{1}{\sqrt{\varepsilon}\lambda(\varepsilon)}(u^{\varepsilon}(t)-\bar{u}(t)),\ t\in [0,T],
\end{equation*}
where  $\lambda(\varepsilon)$ is some deviation scale influencing the asymptotic behavior of $X^{\varepsilon}$.
Concretely, three cases are involved:
\begin{description}
  \item[(1)] The case $\lambda(\varepsilon)=\frac{1}{\sqrt{\varepsilon}}$ provides LDP, which has been proved by \cite{DWZZ}.
  \item[(2)] The case $\lambda(\varepsilon)=1$  provides the central limit theorem (CLT). We will show that $X^\varepsilon$ converges to a solution of a stochastic equation, as $\varepsilon$ decrease to 0 in Section \ref{section3}.
  \item[(3)] To fill in the gap between the CLT scale ($\lambda(\varepsilon)=1$) and the large deviations scale ($\lambda(\varepsilon)=\frac{1}{\sqrt{\varepsilon}}$), we will study the so-called moderate deviation principle (MDP) in Section \ref{section4}. Here, the deviations scale satisfies
      \begin{eqnarray}\label{e-43}
      \lambda(\varepsilon)\rightarrow +\infty,\ \sqrt{\varepsilon}\lambda(\varepsilon)\rightarrow 0\quad {\rm{as}} \ \varepsilon\rightarrow 0.
      \end{eqnarray}
\end{description}

Similar to LDP, MDP arises in the theory of statistical inference naturally, which can provide us with the rate of convergence and a useful method for constructing asymptotic confidence intervals ( see, e.g. \cite{E-1,I-K,K} and references therein). {\color{black} An important tool for studying moderate deviations is the weak convergence approach, which is developed by Dupuis and Ellis in \cite{DE}. For more details on this method, we refer the readers to Bou\'{e}, Dupuis \cite{MP}, Budhiraja, Dupuis \cite{BD} and Budhiraja, Dupuis, Maroulas \cite{BDM08}.
%The key idea is to prove some variational representation formula about
%the Laplace transform of bounded continuous functionals, which will lead to proving an equivalence between
%the Laplace principle and LDP. In particular, for Brownian functionals, an elegant variational
%representation formula has been established by Bou\'{e}, Dupuis \cite{MP} and Budhiraja, Dupuis \cite{BD}.
Recently, a sufficient condition to verify the large deviation criteria of \cite{BDM08} for functionals of Brownian motions is proposed by Matoussi, Sabbagh and Zhang in \cite{MSZ}, which turns out to be more suitable for SPDEs arising from fluid mechanics. Thus, in the present paper, we adopt this new sufficient condition.}
Up to now, there are plenty of results on the moderate deviations for fluid mechanics and other processes. For example, Wang et al. \cite{WZZ} established the CLT and MDP for 2D Navier-Stokes equations driven by multiplicative Gaussian noise in  $C([0,T];H)\cap L^2([0,T];V)$. Further, Dong et al. \cite{DXZZ} consideblack the MDP for 2D Navier-Stokes equations driven by multiplicative L\'{e}vy noises in $D([0,T];H)\cap L^2([0,T];V)$.
In view of the characterization of the super-Brownian motion (SBM) and the Fleming-Viot process (FVP), Fatheddin and Xiong \cite{FX} obtained  MDP for those processes.
Recently, the CLT and MDP for stochastic Burgers equation with viscosity term have been studied by several authors, see \cite{BB,XZ}.

\vskip 0.3cm

As stated above, the aim of this paper is to show two kinds of asymptotic behaviors of $X^\varepsilon$: the CLT and MDP in $L^1([0,T]; L^1(\mathbb{T}^d))$, which provide the exponential decay of small probabilities associated  with the corresponding stochastic dynamical systems with small noise. To our knowledge, this is the first work towards CLT and MDP for kinetic solution of stochastic scalar conservation laws. Due to the lack of viscous term, the kinetic solutions of (\ref{sbe}) are living in a rather irregular space, it is indeed a challenge to establish CLT and MDP for  (\ref{sbe}) with general noise force. To achieve the results, {\color{black}we} divide the proof into two parts.
For the CLT, we need to show that $\frac{1}{\sqrt{\varepsilon}}(u^{\varepsilon}-\bar{u})$ converges to a solution $\bar{u}^1$ of a stochastic equation in $L^1([0,T]; L^1(\mathbb{T}^d))$, as $\varepsilon$ decreases to 0. It's important to point out that although the kinetic formulations are available for both $\frac{1}{\sqrt{\varepsilon}}(u^{\varepsilon}-\bar{u})$ and $\bar{u}^1$ when the initial value of (\ref{sbe-1}) is constant, this convergence cannot be obtained directly by applying doubling variables method.
The essential observation is that the doubling variables method can succeed for two equations only if they are symmetry in some sense.
To overcome this difficulty, we introduce some auxiliary approximations, which are symmetric with the original equations. During the proof process, the vanishing viscosity method and the doubling variables method play an essential role.
 Concerning the MDP, it can be changed to prove that $\frac{1}{\sqrt{\varepsilon}\lambda(\varepsilon)}(u^{\varepsilon}-\bar{u})$ satisfies a large deviation principle in $L^1([0,T]; L^1(\mathbb{T}^d))$ with $\lambda(\varepsilon)$ satisfying (\ref{e-43}). During the proof process of MDP, we encounter the same difficulties as CLT. Therefore, we also need to introduce some suitable parabolic approximation equations.

\smallskip

{\color{black}
To end the introduction, we take $\bar{u}^1$ as an example to explain the difficulty of handling general initial values.
%some difficulties when on the reason why
%it should be pointed out that we can only establish CLT and MDP for  stochastic scalar conservation laws with constant initial values, and the case of general initial values cannot be solved at this stage.
Here, $\bar{u}^1$ is the solution of the limiting equation of $X^\varepsilon$ as $\varepsilon\rightarrow 0$ satisfying
\begin{eqnarray}\label{r-2}
\left\{
  \begin{array}{ll}
  d\bar{u}^1+{\rm{div}} (a(\bar{u})\bar{u}^1)dt=\Phi(\bar{u})dW(t) \ \ {\rm{in}}\ \mathbb{T}^d\times [0,T],\\
\bar{u}^1(x,0)=0 \quad {\rm{on}}\ \mathbb{T}^d,
  \end{array}
\right.
\end{eqnarray}
where $a$ is the derivative of $A$ and $\bar{u}$ is the solution of deterministic conservation laws with initial value $u_0$. Since we want to apply the doubling variables method to $\bar{u}^1$, the kinetic formulation of $\bar{u}^1$ is necessary. Formally (in fact, to derive the kinetic formulation for $\bar{u}^1$, we need to firstly derive the kinetic formulation for viscous approximation $\bar{u}^{1,\eta}$ defined by (\ref{rrr-2}), then let $\eta\rightarrow 0$), by It\^{o} formula and using similar method as in the proof of Proposition 23 in \cite{D-V-1},
it follows that there exists a kinetic measure $m$ such that the kinetic function $f:=I_{\bar{u}^1>\xi}$ satisfies
\begin{align}\notag
&\int^T_0\int_{\mathbb{T}^d}\int_{\mathbb{R}} f(t)\partial_t \varphi(t) dx d\xi dt+\int_{\mathbb{T}^d}\int_{\mathbb{R}} f_0\varphi(0)dxd\xi+\int^T_0\int_{\mathbb{T}^d}\int_{\mathbb{R}}a(\bar{u})\nabla_x  \varphi(x,t,\xi) f(t)d\xi dxdt\\ \notag
&= -\int^T_0\int_{\mathbb{T}^d}\int_{\mathbb{R}} \Big({\rm{div}}  a(\bar{u})\Big)\varphi(x,t,\xi) f(t)d\xi dxdt
+ \int^T_0\int_{\mathbb{T}^d}\Big({\rm{div}}  a(\bar{u})\Big)\varphi(x,\bar{u}^1)\bar{u}^1 dxdt\\
\notag
& - \int^T_0\int_{\mathbb{T}^d}\varphi(x,t,\bar{u}^1)\Phi(\bar{u})dxdW(t)
 -\frac{1}{2}\int^T_0\int_{\mathbb{T}^d}G^2(\bar{u})\int_{\mathbb{R}}\partial_{\xi}\varphi(x,t,\xi)\delta_0(\bar{u}^1-\xi) d\xi dxdt\\
 \label{r-1}
 &+m(\partial_{\xi}\varphi)
\quad a.s..
\end{align}
Notice that the term $ {\rm{div}} a(\bar{u})$ appears in the second line of (\ref{r-1}), but it seems difficult to provide any bound on it when $\bar{u}$ is the solution to the deterministic conservation laws with general initial data $u_0$ based on the result of Section 3.4 in \cite{D-V-1}. Hence, in the present paper, we focus on the special case that $\bar{u}$ is constant which constrains the initial value $u_0$ to be constant.
 %Hence, the Moreover, it looks like $X^{\varepsilon}(t)=\frac{1}{\sqrt{\varepsilon}\lambda(\varepsilon)}(u^{\varepsilon}(t)-\bar{u}(t))$ is not suited to the framework of kinetic solutions, unless $\bar{u}$ is a constant.

%In the future, we aim to develop new approaches to establish CLT and MDP for stochastic conservation laws on $\mathbb{T}^d$ with general initial values, such as continuous initial values, bounded initial values and so on.
}

This paper is organized as follows. The mathematical framework of stochastic conservation laws is in Section 2. We present the proof of the central limit theorem in Section 3. The moderate deviation principle is established in Section 4.

\section{Kinetic solution and hypotheses}\label{section2}
We will follow closely the framework of \cite{D-V-1}.
Let $\|\cdot\|_{L^p}$ denote the norm of usual Lebesgue space $L^p(\mathbb{T}^d)$ for {\color{black}$p\in [1,\infty)$}. {\color{black}In particular, set $H=L^2(\mathbb{T}^d)$ with the corresponding norm $\|\cdot\|_H$ and inner produce $(\cdot,\cdot)$.}
  {\color{black}$C_b$} represents the space of bounded, continuous functions and $C^1_b$ stands for the space of bounded, continuously differentiable functions having bounded first order derivative. Define the function $f(x,t,\xi):=I_{u(x,t)>\xi}$, which is the characteristic function of the subgraph of $u$. We write $f:=I_{u>\xi}$ for short.
 Moreover, denote by the brackets $\langle\cdot,\cdot\rangle$ the duality between $C^{\infty}_c(\mathbb{T}^d\times \mathbb{R})$ and the space of distributions over $\mathbb{T}^d\times \mathbb{R}$. In what follows, with a slight abuse of the notation $\langle\cdot,\cdot\rangle$, we denote the following integral by
\begin{equation*}
\langle F, G \rangle:=\int_{\mathbb{T}^d}\int_{\mathbb{R}}F(x,\xi)G(x,\xi)dxd\xi, \quad F\in L^p(\mathbb{T}^d\times \mathbb{R}), G\in L^q(\mathbb{T}^d\times \mathbb{R}),
\end{equation*}
where $1\leq p < +\infty$, $q:=\frac{p}{p-1}$ is the conjugate exponent of $p$. In particular, when $p=1$, we set $q=\infty$ by convention. For a measure $m$ on the Borel measurable space $\mathbb{T}^d\times[0,T]\times \mathbb{R}$, the shorthand $m(\phi)$ is defined by
\begin{equation*}
m(\phi):=\langle m, \phi \rangle([0,T]):=\int_{\mathbb{T}^d\times[0,T]\times \mathbb{R}}\phi(x,t,\xi)dm(x,t,\xi), \quad  \phi\in C_b(\mathbb{T}^d\times[0,T]\times \mathbb{R}).
\end{equation*}
In the sequel,  the notation $a\lesssim b$ for $a,b\in \mathbb{R}$  means that $a\leq \mathcal{D}b$ for some constant $\mathcal{D}> 0$ independent of any parameters.
\subsection{Hypotheses}
For the flux function $A$ and the coefficient $\Phi$ of (\ref{sbe}), we assume that
\begin{description}
  \item[\textbf{Hypothesis H}] The flux function $A$ belongs to $C^2(\mathbb{R};\mathbb{R}^d)$ and its derivative $a$ has at most polynomial growth. That is, there exist constants $C>0, p\geq1$ such that
     \begin{eqnarray}\label{qeq-22}
     |a(\xi)-a(\zeta)|\leq \Gamma(\xi,\zeta)|\xi-\zeta|, \quad \Gamma(\xi,\zeta)=C(1+|\xi|^{p-1}+|\zeta|^{p-1}).
      \end{eqnarray}

  The map $\Phi(u): H\rightarrow H$ is defined by $\Phi(u) e_k:=g_k(\cdot, u)$, where $(e_k)_{k\geq 1}$ is a complete orthonormal base in the Hilbert space $H$ and each $g_k(\cdot,u)$ is a regular function on $\mathbb{T}^d$. More precisely, we assume that $g_k\in C(\mathbb{T}^d\times \mathbb{R})$ satisfying the following bounds
\begin{eqnarray}\label{equ-28}
G^2(x,u):=\sum_{k\geq 1}|g_k(x,u)|^2&\leq& D_0(1+|u|^2),\\
\label{equ-29}
\sum_{k\geq 1}|g_k(x,u)-g_k(y,v)|^2&\leq& D_1\Big(|x-y|^2+{|u-v|^2}\Big),
\end{eqnarray}
for $x, y\in \mathbb{T}^d, u,v\in \mathbb{R}$.
\end{description}

%Since $\|g_k\|_{H}\leq\|g_k\|_{C(\mathbb{T}^d)}$, we deduce that $\Phi(u)\in \mathcal{L}_2(U,H)$, for each $u\in \mathbb{R}$.
%Moreover, it follows from (\ref{equ-28}) and (\ref{equ-29}) that
%\begin{eqnarray*}\label{equ-30}
%\|\Phi(u)\|^2_{\mathcal{L}_2(U,H)}&\leq& D_0(1+\|u\|^2_H),\\
%\label{equ-30-1}
%\|\Phi(u)-\Phi(v)\|^2_{\mathcal{L}_2(U,H)}&\leq&  D_1\|u-v\|^2_H.
%\end{eqnarray*}
%\begin{remark}
%In order to obtain the small time large deviations, our assumptions are stronger than those used by \cite{D-V-1} to prove the existence and uniqueness of (\ref{P-19}).
%\end{remark}

%Based on the above notations, equation (\ref{sbe}) can be rewritten as
%\begin{eqnarray}\label{sbe}
%\left\{
%  \begin{array}{ll}
%  du(x,t)+ {\rm{div}} (A(u(x,t)))dt=\sum_{k\geq 1}g_k(x,u(x,t)) d\beta_k(t) \quad {\rm{in}}\ \mathbb{T}^d\times (0,T],\\
%u(\cdot,0)=1 \quad {\rm{on}}\ \mathbb{T}^d.
%  \end{array}
%\right.
%\end{eqnarray}

{\textbf{From now on and in the sequel, we always assume Hypothesis H is in force.}}
\subsection{Kinetic solution}
  Keeping in mind that we are working on the stochastic basis $(\Omega,\mathcal{F},P,\{\mathcal{F}_t\}_{t\in [0,T]},(\beta_k(t))_{k\in\mathbb{N}})$.
\begin{dfn}(Kinetic measure)\label{dfn-3}
 A map $m$ from $\Omega$ to the set of non-negative, finite measures over $\mathbb{T}^d\times [0,T]\times\mathbb{R}$ is said to be a kinetic measure, if
\begin{description}
  \item[1.] $ m $ is measurable, that is, for each $\phi\in C_b(\mathbb{T}^d\times [0,T]\times \mathbb{R})$, ${\color{black}\langle m, \phi \rangle([0,T])}: \Omega\rightarrow \mathbb{R}$ is measurable,
  \item[2.] $m$ vanishes for large $\xi$, i.e.,
\begin{eqnarray}\label{equ-37}
\lim_{R\rightarrow +\infty}E[m(\mathbb{T}^d\times [0,T]\times B^c_R)]=0,
\end{eqnarray}
where $B^c_R:=\{\xi\in \mathbb{R}, |\xi|\geq R\}$
  \item[3.] for every $\phi\in C_b(\mathbb{T}^d\times \mathbb{R})$, the process
\begin{equation*}
(\omega,t)\in\Omega\times[0,T]\mapsto \int_{\mathbb{T}^d\times [0,t]\times \mathbb{R}}\phi(x,\xi)dm(x,s,\xi)\in\mathbb{R}
\end{equation*}
is pblackictable.
\end{description}
\end{dfn}
Let $\mathcal{M}^+_0(\mathbb{T}^d\times [0,T]\times \mathbb{R})$ be the space of all bounded, nonnegative random measures $m$ satisfying (\ref{equ-37}).

\begin{dfn}(Kinetic solution)\label{dfn-1}
 A measurable function {\color{black}$u: \Omega\times\mathbb{T}^d\times [0,T]\rightarrow \mathbb{R}$} is called a kinetic solution to (\ref{sbe}), if
\begin{description}
  \item[1.] $(u(t))_{t\in[0,T]}$ is pblackictable,
  \item[2.] for any $p\geq1$, there exists $C_p\geq0$ such that
\begin{equation*}
E\left(\underset{0\leq t\leq T}{{\rm{ess\sup}}}\ \|u(t)\|^p_{L^p(\mathbb{T}^d)}\right)\leq C_p,
\end{equation*}
\item[3.] there exists a kinetic measure $m$ such that $f:= I_{u>\xi}$ satisfies the following
\begin{eqnarray}\notag
&&\int^T_0\langle f(t), \partial_t \varphi(t)\rangle dt+\langle f_0, \varphi(0)\rangle +\int^T_0\langle f(t), a(\xi)\cdot \nabla \varphi (t)\rangle dt\\ \notag
&=& -\sum_{k\geq 1}\int^T_0\int_{\mathbb{T}^d} g_k(x, u(x,t))\varphi (x,t,u(x,t))dxd\beta_k(t) \\
\label{P-21}
&& -\frac{1}{2}\int^T_0\int_{\mathbb{T}^d}\partial_{\xi}\varphi (x,t,u(x,t))G^2(x,u(x,t))dxdt+ m(\partial_{\xi} \varphi), \ a.s. ,
\end{eqnarray}
for all $\varphi\in C^1_c(\mathbb{T}^d\times [0,T]\times \mathbb{R})$, where $u(t)=u(\cdot,t,\cdot)$, $f_0=I_{1>\xi}$ and $a(\xi):=A'(\xi)$.
\end{description}
\end{dfn}

Let $(X,\lambda)$ be a finite measure space. For some measurable function $u: X\rightarrow \mathbb{R}$, define $f: X\times \mathbb{R}\rightarrow [0,1]$ by $f(z,\xi)=I_{u(z)>\xi}$ a.e.
we use $\bar{f}:=1-f$ to denote its conjugate function. Define $\Lambda_f(z,\xi):=f(z,\xi)-I_{0>\xi}$, which can be viewed as a correction to $f$. Note that $\Lambda_f$ is integrable on $X\times \mathbb{R}$ if $u$ is.
\vskip 0.3cm

It is shown in \cite{D-V-1} that almost surely, for each kinetic solution $u$, the function $f=I_{u(x,t)>\xi}$ admits left and right weak limits at any point $t\in[0,T]$, and the weak form (\ref{P-21}) satisfied by a kinetic solution can be strengthened to be weak only respect to $x$ and $\xi$. More precisely,  the following results are obtained.
\begin{prp}(\cite{D-V-1}, Left and right weak limits)\label{prp-3} Let $u$ be a kinetic solution to (\ref{sbe}). Then $f=I_{u(x,t)>\xi}$  admits, almost surely, left and right limits respectively at every point $t\in [0,T]$. More precisely, for any  $t\in [0,T]$, there exist functions $f^{t\pm}$ on $\Omega\times \mathbb{T}^d\times \mathbb{R}$ such that $P-$a.s.
\begin{eqnarray*}%\label{e-50}
\langle f(t-r),\varphi\rangle\rightarrow \langle f^{t-},\varphi\rangle
\end{eqnarray*}
and
\begin{eqnarray*}%\label{e-51}
\langle f(t+r),\varphi\rangle\rightarrow \langle f^{t+},\varphi\rangle
\end{eqnarray*}
as $r\rightarrow 0$ for all $\varphi\in C^1_c(\mathbb{T}^d\times \mathbb{R})$. Moreover, almost surely,
\begin{equation*}
\langle f^{t+}-f^{t-}, \varphi\rangle=-\int_{\mathbb{T}^d\times[0,T]\times \mathbb{R}}\partial_{\xi}\varphi(x,\xi)I_{\{t\}}(s)dm(x,s,\xi).
\end{equation*}
In particular, almost surely, the set of $t\in [0,T]$ fulfilling $f^{t+}\neq f^{t-}$ is countable.
\end{prp}
For the function $f=I_{u(x,t)>\xi}$ in Proposition \ref{prp-3}, define $f^{\pm}$ by $f^{\pm}(t)=f^{t \pm}$, $t\in [0,T]$. Since we are dealing with the filtration associated to Brownian motion, both $f^{+}$ and $f^{-}$ are  clearly pblackictable as well. Also $f=f^+=f^-$ almost everywhere in time and we can take any of them in an integral with respect to the Lebesgue measure or in a stochastic integral. However, if the integral is with respect to a measure, typically a kinetic measure in this article, the integral is not well-defined for $f$ and may differ if one chooses $f^+ $ or $f^-$.

%As discussed above, with the aid of Proposition \ref{prp-3}, the weak form (\ref{P-22}) satisfied by a generalized kinetic solution can be strengthened to be weak only respect to $x$ and $\xi$. Concretely,
With the aid of Proposition \ref{prp-3}, the following result was proved in  \cite{D-V-1}.
 \begin{lemma}\label{lem-1}
 The weak form (\ref{P-21}) satisfied by $f= I_{u>\xi}$ can be strengthened to be weak only respect to $x$ and $\xi$. Concretely, for all $t\in [0,T)$ and $\varphi\in C^1_c(\mathbb{T}^d\times \mathbb{R})$, $f= I_{u>\xi}$ satisfies
 \begin{eqnarray}\notag
\langle f^+(t),\varphi\rangle&=&\langle f_{0}, \varphi\rangle+\int^t_0\langle f(s), a(\xi)\cdot \nabla \varphi\rangle ds\\
\notag
&&+\sum_{k\geq 1}\int^t_0\int_{\mathbb{T}^d}\int_{\mathbb{R}}g_k(x,\xi)\varphi(x,\xi)d\nu_{x,s}(\xi)dxd\beta_k(s)\\
\label{qq-17}
&& +\frac{1}{2}\int^t_0\int_{\mathbb{T}^d}\int_{\mathbb{R}}\partial_{\xi}\varphi(x,\xi)G^2(x,\xi)d\nu_{x,s}(\xi)dxds- \langle m,\partial_{\xi} \varphi\rangle([0,t]), \ a.s.,
\end{eqnarray}
where $\nu_{x,s}(\xi)=-\partial_{\xi}f(x,s,\xi)=\delta_{u(x,s)=\xi}$ and we set $f^+(T)=f(T)$.
 \end{lemma}

\begin{remark}  By making  modification of the proof of Lemma \ref{lem-1}, we have for all $t\in (0,T]$ and $\varphi\in C^1_c(\mathbb{T}^d\times \mathbb{R})$, $f= I_{u>\xi}$ satisfies
   \begin{eqnarray}\notag
\langle f^-(t),\varphi\rangle&=&\langle f_{0}, \varphi\rangle+\int^t_0\langle f(s), a(\xi)\cdot \nabla \varphi\rangle ds\\
\notag
&&+\sum_{k\geq 1}\int^t_0\int_{\mathbb{T}^d}\int_{\mathbb{R}}g_k(x,\xi)\varphi(x,\xi)d\nu_{x,s}(\xi)dxd\beta_k(s)\\
\label{e-80}
&& +\frac{1}{2}\int^t_0\int_{\mathbb{T}^d}\int_{\mathbb{R}}\partial_{\xi}\varphi(x,\xi)G^2(x,\xi)d\nu_{x,s}(\xi)dxds- \langle m,\partial_{\xi} \varphi\rangle([0,t)), \ a.s.,
\end{eqnarray}
and we set $ f^-(0)=f_0$.
\end{remark}

\subsection{Global well-posedness of (\ref{sbe})}
According to \cite{D-V-1}, we have the following results of (\ref{sbe}).
\begin{thm}\label{thm-4}
(Existence, Uniqueness) Assume Hypothesis H holds. Then there is a unique kinetic solution $u$ to equation (\ref{sbe}), and there exist $u^+$ and $u^{-}$, representatives of $u$
 such that for all $t\in[0,T]$, $f^{\pm}(x,t,\xi)=I_{u^{\pm}(x,t)>\xi}\ a.s.$ for $a.e.\  (x,t,\xi)$.
\end{thm}
\begin{cor}(Continuity in time) Assume Hypothesis H is in force, then for every $p\in [1,+\infty)$, the kinetic solution $u$ to equation (\ref{sbe}) has a representative in $L^p(\Omega; L^{\infty}(0,T;L^p(\mathbb{T}^d)))$ with almost sure continuous trajectories in $L^p(\mathbb{T}^d)$.

\end{cor}

\section{Central limit theorem}\label{section3}
{\color{black}This part is devoted to} proving a central limit theorem for stochastic scalar conservation laws (\ref{sbe}).

For any $\varepsilon\in (0,1)$, let $u^{\varepsilon}$ be the unique kinetic solution to (\ref{sbe-1}). As the parameter $\varepsilon$ approaches zero, the solution $u^{\varepsilon}$ will tend to the solution of the following deterministic {\color{black}scalar conservation laws}
\begin{eqnarray}\label{e-2}
\left\{
  \begin{array}{ll}
  d\bar{u}+ {\rm{div}} (A(\bar{u}))dt=0 \quad {\rm{in}}\ \mathbb{T}^d\times {\color{black}[0,T]},\\
\bar{u}(x,0)=1 \quad {\rm{on}}\ \mathbb{T}^d.
  \end{array}
\right.
\end{eqnarray}
For (\ref{e-2}), we claim that it admits a unique solution $\bar{u}\equiv 1$. Indeed, for any $\eta>0$,
let us consider the following equation
\begin{eqnarray}\label{eeta-2}
\left\{
  \begin{array}{ll}
  d\bar{u}^{\eta}+ {\rm{div}} (A(\bar{u}^{\eta}))dt=\eta\Delta\bar{u}^{\eta}dt \quad {\rm{in}}\ \mathbb{T}^d\times {\color{black}[0,T]},\\
\bar{u}^{\eta}(x,0)=1 \quad {\rm{on}}\ \mathbb{T}^d.
  \end{array}
\right.
\end{eqnarray}
Clearly, (\ref{eeta-2}) has a unique strong solution $\bar{u}^{\eta}(t,x)\equiv1$, for any $(t,x)\in [0,T]\times\mathbb{T}^d$. As a result, by vanishing viscosity method, ({\color{black}}\ref{e-2}) has a unique kinetic solution $\bar{u}(t,x)\equiv1$, for any $(t,x)\in [0,T]\times\mathbb{T}^d$.

\smallskip
In the following, we will investigate the fluctuation behaviour of $\frac{1}{\sqrt{\varepsilon}}(u^{\varepsilon}-\bar{u})$.
Let $\bar{u}^1$ be the solution of the following SPDE
\begin{eqnarray}\label{u1}
\left\{
  \begin{array}{ll}
  d\bar{u}^1+{\rm{div}} (a(\bar{u})\bar{u}^1)dt=\Phi(\bar{u})dW(t) \ \ {\rm{in}}\ \mathbb{T}^d\times {\color{black}[0,T]},\\
\bar{u}^1(x,0)=0 \quad {\rm{on}}\ \mathbb{T}^d.
  \end{array}
\right.
\end{eqnarray}
 Taking into account $\bar{u}\equiv1$, equation (\ref{u1}) turns out to be conservation laws driven by an additive noise with flux function $A(\xi)=a(1)\xi$. It follows from Theorem \ref{thm-4} that (\ref{u1}) admits a unique kinetic solution $\bar{u}^1$.

% Moreover,
%\begin{equation}\label{eq-1}
%\bar{u}^1(x,t)=\int_0^t\sum_{k\geq1}g_k(x-\bar{u}(t-s),\bar{u})d\beta_k(s).
%\end{equation}
%It is readily to see that the law of $\bar{u}^1$ is Gaussian.

We will show that $\frac{1}{\sqrt{\varepsilon}}(u^{\varepsilon}-\bar{u})$ converges to $\bar{u}^1$ in the space $L^1([0,T];L^1(\mathbb{T}^d))$. Although the kinetic formulations for both $\frac{1}{\sqrt{\varepsilon}}(u^{\varepsilon}-\bar{u})$ and $\bar{u}^1$ are both available, we cannot prove this convergence directly by doubling variables method. The difficulty lies in the estimation of $\tilde{K}_1$ in the following Proposition \ref{prp-2}. Concretely, when applying the doubling variables method directly to $\frac{1}{\sqrt{\varepsilon}}(u^{\varepsilon}-\bar{u})$ and $\bar{u}^1$, the term $\tilde{K}_1$ cannot converge to 0 as $\varepsilon\rightarrow 0$. The essential reason is that the doubling variables method can succeed for two equations only if they are symmetry in some sense.
To overcome this difficulty, we will introduce some auxiliary approximation processes, which are symmetric with the original equations.

%For $\frac{1}{\sqrt{\varepsilon}}(u^{\varepsilon}-\bar{u})$,
%it seems that we could not (at least very difficult) to get the bound of $E\|\frac{1}{\sqrt{\varepsilon}}(u^{\varepsilon}-\bar{u})\|^p_{L^p(\mathbb{T}^d\times [0,T])}$ uniformly with respect to $\varepsilon>0$. As a consequence, the doubling variable method cannot be applied directly to $\frac{1}{\sqrt{\varepsilon}}(u^{\varepsilon}-\bar{u})$ and $\bar{u}^1$,
%although kinetic formulations for both $\frac{1}{\sqrt{\varepsilon}}(u^{\varepsilon}-\bar{u})$ and $\bar{u}^1$ are available.

Consider the following viscosity approximating process $u^{\varepsilon,\eta}$ of $u^{\varepsilon}$ given by
\begin{eqnarray}\label{uepe}
\left\{
  \begin{array}{ll}
  du^{\varepsilon,\eta}+ {\rm{div}} (A(u^{\varepsilon,\eta})) dt =\eta\Delta u^{\varepsilon,\eta}dt+\sqrt{\varepsilon}\Phi(u^{\varepsilon,\eta})dW(t)\ \ {\rm{in}}\ \ \mathbb{T}^d\times[0,T],\\
u^{\varepsilon,\eta}(x,0)=1 \quad {\rm{on}}\ \mathbb{T}^d.
  \end{array}
\right.
\end{eqnarray}
%utilize the vanishing viscosity approximation for, we consider $u^{\varepsilon,\eta}
{\color{black}Referring to Proposition 23 in \cite{D-V-1},} (\ref{uepe}) has a unique kinetic solution $u^{\varepsilon,\eta}$. Moreover, by It\^{o} formula, it's readily to deduce that for any $q\geq 2$,
\begin{eqnarray}\label{r-20}
 \sup_{\varepsilon\in (0,1),\eta>0} \sup_{t\in [0,T]}E\|u^{\varepsilon,\eta}\|^q_{L^q( \mathbb{T}^d)}\leq C(p,T).
\end{eqnarray}
In fact, (\ref{r-20}) is a direct consequence of the following Lemma \ref{lem-2}.

We also need the viscosity approximating process of $\bar{u}^1$, which can be written as
\begin{eqnarray}\label{rrr-2}
\left\{
  \begin{array}{ll}
  d\bar{u}^{1,\eta}+ {\rm{div}} (a(\bar{u})\bar{u}^{1,\eta})dt=\eta\Delta \bar{u}^{1,\eta}dt+\Phi(\bar{u})dW(t) \ \ {\rm{in}}\ \mathbb{T}^d\times[0,T],\\
\bar{u}^{1,\eta}(x,0)=1  \quad {\rm{on}}\ \mathbb{T}^d.
  \end{array}
\right.
\end{eqnarray}
Since $\bar{u}\equiv 1$, due to \cite{D-V-1}, (\ref{rrr-2}) admits a unique solution $\bar{u}^{1,\eta}$. Moreover, referring to Proposition 24 in \cite{D-V-1}, we have
\begin{eqnarray}\label{r-3}
  \lim_{\eta\rightarrow 0}E\|\bar{u}^{1,\eta}-\bar{u}^{1}\|_{L^1([0,T];L^1( \mathbb{T}^d))}=0.
\end{eqnarray}

With the above approximation processes $u^{\varepsilon,\eta}, \bar{u}^{\eta}$ and $ \bar{u}^{1,\eta}$, it gives
\begin{align}\notag
&E\Big\|\frac{u^{\varepsilon}-\bar{u}}{\sqrt{\varepsilon}}-\bar{u}^1\Big\|_{L^1([0,T];L^1(\mathbb{T}^d))}\\\notag
&\leq E\Big\|\frac{u^{\varepsilon}-\bar{u}}{\sqrt{\varepsilon}}-\frac{u^{\varepsilon,\eta}-\bar{u}^{\eta}}{\sqrt{\varepsilon}}\Big\|_{L^1([0,T];L^1(\mathbb{T}^d))}
+E\Big\|\frac{u^{\varepsilon,\eta}-\bar{u}^{\eta}}{\sqrt{\varepsilon}}-\bar{u}^{1,\eta}\Big\|_{L^1([0,T];L^1(\mathbb{T}^d))} \\
\label{3terms}
& +E\|\bar{u}^{1,\eta}-\bar{u}^1\|_{L^1([0,T];L^1(\mathbb{T}^d))}.
\end{align}
With the help of (\ref{r-3}), it remains to handle the other two terms.
Before giving the proof process, we need a priori estimate.
Recall that $u^{\varepsilon,\eta}$ is the unique solution of (\ref{uepe}) and $\bar{u}^{\eta}$ is the unique solution of (\ref{eeta-2}). We have the following result relating the difference between $u^{\varepsilon,\eta}$ and $\bar{u}^{\eta}$.
{\color{black}
\begin{lemma}\label{lem-2}
%$\eta>0$ and for any $p\geq 1$, there exists constant $C=C(T,p, D_0)$ such that
%\begin{align}\label{l2}
%\sup_{t\in[0,T]}E\|u^{\varepsilon,\eta}(t)-\bar{u}^{\eta}(t)\|_{H}^{2p}+\eta E\int_0^T\|\nabla (u^{\varepsilon,\eta}(s)-\bar{u}^{\eta})\|^{2p}_{H}ds\leq C\varepsilon^p.
%\end{align}%Moreover,
For any $\varepsilon\in (0,1)$ and $q\geq 2$, there exists a constant $C$ depending on $q$ and $T$ such that
\begin{align}\label{l4}
\sup_{t\in [0,T]}\sup_{\eta>0} E \|u^{\varepsilon,\eta}(t)-\bar{u}^{\eta}(t)\|^q_{L^q(\mathbb{T}^d)}
\leq C(q,T)\varepsilon^{\frac{p}{2}}.
\end{align}
\end{lemma}
}
\begin{proof}
Let $w_0^{\varepsilon,\eta}:=u^{\varepsilon,\eta}-\bar{u}^{\eta}$, by using $\bar{u}^{\eta}=1$, it follows from (\ref{eeta-2}) and (\ref{uepe}) that $w_0^{\varepsilon,\eta}$ satisfies
\begin{equation*}
dw_0^{\varepsilon,\eta}+ {\rm{div}} ( A(w_0^{\varepsilon,\eta}+1))dt=\eta\Delta w_0^{\varepsilon,\eta}dt+\sqrt{\varepsilon}\Phi(u^{\varepsilon,\eta})dW(t),
\end{equation*}
with $w_0^{\varepsilon,\eta}(0)=0$.

Let $\varphi(\xi)=|\xi|^{q}$ and $\psi(x)=1$. To get the proof of (\ref{l4}), we follow the approach of Proposition 5.1 in \cite{DHV}
and introduce functions $\varphi_n\in C^2(\mathbb{R})$ that approximate $\varphi$ and
have quadratic growth at infinity. Concretely, let
\begin{eqnarray*}
 \varphi_n(\xi)=\left\{
                  \begin{array}{ll}
                    |\xi|^q, & |\xi|\leq n, \\
                    n^{q-2}\Big[\frac{q(q-1)}{2}\xi^2-q(q-2)n|\xi|+\frac{(q-1)(q-2)}{2}n^2\Big], &|\xi|> n.
                  \end{array}
                \right.
\end{eqnarray*}
Clearly, by the definition of $\varphi_n$, we have
\begin{eqnarray}\notag
  &&|\xi\varphi'_n(\xi)|\leq q\varphi_n(\xi), \quad
|\varphi'_n(\xi)|\leq q(1+\varphi_n(\xi)),\\
\label{r-6}
 &&|\varphi'_n(\xi)|\leq |\xi|\varphi''_n(\xi),\quad
\xi^2\varphi''_n(\xi)\leq q(q-1)\varphi_n(\xi),\quad
(\varphi''_n(\xi))^{\frac{q}{q-2}}\leq C(q)\varphi_n(\xi),
\end{eqnarray}
for all $\xi\in \mathbb{R}$, $n\in \mathbb{N}$, $q\in [2,\infty)$.
Then, functions $\varphi_n$ and $\psi(x)=1$ satisfy conditions requiblack by Proposition A.1 in \cite{DHV}. Hence, by Proposition A.1 in \cite{DHV}, it follows that
\begin{align*}
  \int_{\mathbb{T}^d}\varphi_n(w_0^{\varepsilon,\eta}(t))dx
=&-\int^t_0\int_{\mathbb{T}^d}\varphi'_n(w_0^{\varepsilon,\eta}) {\rm{div}} (A(w_0^{\varepsilon,\eta}+1))dxds
+\eta\int^t_0\int_{\mathbb{T}^d}\varphi'_n(w_0^{\varepsilon,\eta})\Delta w_0^{\varepsilon,\eta}dxds\\
+&\sqrt{\varepsilon}\sum_{k\geq 1}\int^t_0\int_{\mathbb{T}^d}\varphi'_n(w_0^{\varepsilon,\eta})g_k(x,w_0^{\varepsilon,\eta}+1)dxd\beta_k(s)
+\frac{\varepsilon}{2}\int^t_0\int_{\mathbb{T}^d}\varphi''_n(w_0^{\varepsilon,\eta})G^2(w_0^{\varepsilon,\eta}+1)dxds\\
=:& I_1+I_2+I_3+I_4.
\end{align*}
Setting $H(\xi)=\int^{\xi}_0\varphi''_n(\zeta)A(\zeta+1)d\zeta$, $I_1$ vanishes due to the boundary conditions.
By integration by parts, $I_2$ is nonpositive. Clearly, $EI_3=0$.
The term $I_4$ can be estimated by (\ref{equ-28}) and (\ref{r-6}),
\begin{align*}
  I_4\leq & \frac{\varepsilon}{2}D_0\int^t_0\int_{\mathbb{T}^d}\varphi''_n(w_0^{\varepsilon,\eta})(2+|w_0^{\varepsilon,\eta}|^2)dxds\\
\leq &\frac{\varepsilon q(q-1)}{2}D_0\int^t_0\int_{\mathbb{T}^d}\varphi_n(w_0^{\varepsilon,\eta})dxds
+\varepsilon D_0\int^t_0\int_{\mathbb{T}^d}\varphi''_n(w_0^{\varepsilon,\eta})dxds\\
\leq &\frac{\varepsilon q(q-1)}{2}D_0\int^t_0\int_{\mathbb{T}^d}\varphi_n(w_0^{\varepsilon,\eta})dxds
+C(q,D_0)\varepsilon^{\frac{q}{2}}+\int^t_0\int_{\mathbb{T}^d}\varphi_n(w_0^{\varepsilon,\eta})dxds.
\end{align*}
Hence,
\begin{align*}
\sup_{t\in [0,T]} E \int_{\mathbb{T}^d}\varphi_n(w_0^{\varepsilon,\eta}(t))dx
\leq C(q,T,D_0)\varepsilon^{\frac{q}{2}}+\Big[ \frac{\varepsilon q(q-1)}{2}D_0+1\Big]E\int^t_0\int_{\mathbb{T}^d}\varphi_n(w_0^{\varepsilon,\eta})dxds.
\end{align*}
By Gronwall inequality, we get
\begin{align*}
\sup_{t\in [0,T]} E \int_{\mathbb{T}^d}\varphi_n(w_0^{\varepsilon,\eta}(t))dx
\leq C(q,T)\varepsilon^{\frac{q}{2}}.
\end{align*}
Let $n\rightarrow \infty$, we have
\begin{align*}
\sup_{t\in [0,T]} E \|w_0^{\varepsilon,\eta}(t)\|^q_{L^q(\mathbb{T}^d)}
\leq C(q,T)\varepsilon^{\frac{q}{2}}.
\end{align*}

\end{proof}
Based on Lemma \ref{lem-2} and by Theorem 24 in \cite{D-V-1}, we get the following result.
\begin{cor}\label{cor-1}
  For any $\varepsilon\in (0,1)$ and $q\geq 2$, there exists a constant $C$ depending on $q$ and $T$ such that
\begin{align}\label{r-14}
E \|u^{\varepsilon}-\bar{u}\|^q_{L^q([0,T];L^q( \mathbb{T}^d))}
=\lim_{\eta\rightarrow 0}E \|u^{\varepsilon,\eta}-\bar{u}^{\eta}\|^q_{L^q([0,T];L^q( \mathbb{T}^d))}
\leq C(q,T)\varepsilon^{\frac{q}{2}}.
\end{align}
\end{cor}
\smallskip

Denote by $v^{\varepsilon}:=\frac{u^{\varepsilon}-\bar{u}}{\sqrt{\varepsilon}}$ and $v^{\varepsilon,\eta}:=\frac{u^{\varepsilon,\eta}-\bar{u}^{\eta}}{\sqrt{\varepsilon}}$. We derive from (\ref{sbe-1}) and (\ref{e-2}) that $v^{\varepsilon}$ satisfies
\begin{align}\label{vep}
dv^{\varepsilon}+\frac{1}{\sqrt{\varepsilon}} {\rm{div}} \Big(A(\sqrt{\varepsilon}v^{\varepsilon}+1)-A(1)\Big)dt&=\Phi(\sqrt{\varepsilon}v^{\varepsilon}+1)dW(t),
\end{align}
with initial value $v^{\varepsilon}(0)=0$.
Moreover, it follows from (\ref{eeta-2}) and (\ref{uepe}) that $v^{\varepsilon,\eta}$ fulfills
\begin{equation}\label{vepe}
dv^{\varepsilon,\eta}+\frac{1}{\sqrt{\varepsilon}} {\rm{div}} \Big(A(\sqrt{\varepsilon}v^{\varepsilon,\eta}+1)-A(1)\Big)dt=\eta\Delta v^{\varepsilon,\eta}dt+\Phi(\sqrt{\varepsilon}v^{\varepsilon,\eta}+1)dW(t),
\end{equation}
with initial value $v^{\varepsilon,\eta}(0)=0$.

With the help of Lemma \ref{lem-2} and Corollary \ref{cor-1}, it holds that
{\color{black}\begin{lemma}\label{lem-3}
For any $q\geq 2$, there exists a constant $C$ depending on $q$ and $T$ such that
\begin{align}\label{r-4}
\sup_{\varepsilon\in (0,1)}E\|v^{\varepsilon}\|^q_{L^q([0,T];L^q( \mathbb{T}^d))}\leq C(q,T),\\
\label{r-5}
\sup_{\varepsilon\in (0,1),\eta>0}E\|v^{\varepsilon,\eta}\|^q_{L^q([0,T];L^q( \mathbb{T}^d))}\leq C(q,T).
\end{align}
Moreover, by the definitions of $v^{\varepsilon}$ and $v^{\varepsilon,\eta}$, we have
\begin{align}\label{r-15}
\sup_{\varepsilon\in (0,1)}E\|\sqrt{\varepsilon}v^{\varepsilon}\|^q_{L^q([0,T];L^q( \mathbb{T}^d))}\leq& C(q,T),\\
\label{r-16}
\sup_{\varepsilon\in (0,1),\eta>0}E\|\sqrt{\varepsilon}v^{\varepsilon,\eta}\|^q_{L^q([0,T];L^q( \mathbb{T}^d))}\leq & C(q,T).
\end{align}

\end{lemma}
}
%Moreover, by Ito formula, we have
%\begin{equation}\label{r-6}
%\sup_{\varepsilon,\eta\in (0,1)}E\| v^{\varepsilon,\eta}\|^p_{L^p([0,T]\times \mathbb{T}^d)}\leq C(p).
%\end{equation}

In view of (\ref{r-4}) and (\ref{r-5}), we can apply the doubling variables method to show the following result, i.e., the convergence of the first term in the righthand side of (\ref{3terms}).
\begin{prp}\label{prp-2}
We have
\begin{equation*}
\lim_{\eta\rightarrow0}\sup_{\varepsilon\in(0,1)}E\Big\|\frac{u^{\varepsilon}-\bar{u}}{\sqrt{\varepsilon}}-\frac{u^{\varepsilon,\eta}-\bar{u}^{\eta}}{\sqrt{\varepsilon}}\Big\|_{L^1([0,T];L^1(\mathbb{T}^d))}=0.
\end{equation*}
\end{prp}

\begin{proof}

%Clearly, it holds that
%\begin{equation*}\label{vl1}
%\sup_{\varepsilon\in(0,1]}E\|v^{\varepsilon}\|_{L^1([0,T]\times\mathbb{T}^d)}<\infty.
%\end{equation*}
Denote by $f_1(x,t,\xi):=I_{v^{\varepsilon}(x,t)>\xi}$ and $f_2(y,t,\zeta):=I_{v^{\varepsilon,\eta}(y,t)>\zeta}$ with the corresponding kinetic measures $m^{\varepsilon}_1$, $m^{\varepsilon,\eta}_2$ and initial values $f_{1,0}=I_{0>\xi}$, $f_{2,0}=I_{0>\zeta}$, respectively. {\color{black}Let $\mathbb{T}^d_x$ represent a $d-$dimensional torus with respect to $x$ and
 $\mathbb{R}_{\xi}$ stand for the real line with respect to $\xi$. Similar to (\ref{qq-17}), for any $\varphi_1\in C^{1}_c(\mathbb{T}^d_x\times \mathbb{R}_{\xi})$} and $t\in [0,T)$, we get
\begin{eqnarray*}\notag
\langle f^+_1(t),\varphi_1\rangle&=&\langle f_{1,0}, \varphi_1\rangle+\int^t_0\langle f_1(s), a(\sqrt{\varepsilon}\xi+1)\cdot \nabla_x \varphi_1(s)\rangle ds\\
 && +\frac{1}{2}\int^t_0\int_{\mathbb{T}^d}\int_{\mathbb{R}}\partial_{\xi}\varphi_1(x,\xi)G^2(x,\sqrt{\varepsilon}\xi+1)d\nu^{1,\varepsilon}_{x,s}(\xi)dxds- \langle m^{\varepsilon}_1,\partial_{\xi} \varphi_1\rangle([0,t])\\
 && +\sum_{k\geq 1}\int^t_0\int_{\mathbb{T}^d}\int_{\mathbb{R}}g_k(x,\sqrt{\varepsilon}\xi+1)\varphi_1(x,\xi)d\nu^{1,\varepsilon}_{x,s}(\xi)dxd\beta_k(s)
, \quad  a.s.,
\end{eqnarray*}
 where $\nu^{1,\varepsilon}_{x,s}(\xi)=-\partial_{\xi}f_1(x,s,\xi)=\delta_{v^{\varepsilon}(x,s)=\xi}$.
Moreover, according to (\ref{vepe}), for any
$\varphi_2\in C^{1}_c(\mathbb{T}^d_y\times \mathbb{R}_{\zeta})$ and $t\in [0,T)$, it follows that
\begin{eqnarray*}\notag
\langle \bar{f}^+_2(t),\varphi_2\rangle&=&\langle \bar{f}_{2,0}, \varphi_2\rangle +\int^t_0\langle \bar{f}_2(s), a(\sqrt{\varepsilon}\zeta+1)\cdot \nabla_y \varphi_2(s)+\eta \Delta \varphi_2(s) \rangle ds \\
 && -\frac{1}{2}\int^t_0\int_{\mathbb{T}^d}\int_{\mathbb{R}}\partial_{\zeta}\varphi_2(y,\zeta)G^2(y,\sqrt{\varepsilon}\zeta+1)d\nu^{2,\varepsilon,\eta}_{y,s}(\zeta)dyds+\langle m^{\varepsilon,\eta}_2,\partial_{\zeta} \varphi_2\rangle([0,t])\\
&& -\sum_{k\geq 1}\int^t_0\int_{\mathbb{T}^d}\int_{\mathbb{R}}\tilde{g}_k(y,\sqrt{\varepsilon}\zeta+1)\varphi_2(y,\zeta)d\nu^{2,\varepsilon,\eta}_{y,s}(\zeta)dyd\beta_k(s)\quad a.s.,
\end{eqnarray*}
where $\nu^{2,\varepsilon}_{y,s}(\zeta)=\partial_{\zeta}\bar{f}_2(y,s,\zeta)=\delta_{v^{\varepsilon,\eta}(y,s)=\zeta}$.

Denote the duality distribution over $\mathbb{T}^N_x\times \mathbb{R}_\xi\times \mathbb{T}^N_y\times \mathbb{R}_\zeta$ by $\langle\langle\cdot,\cdot\rangle\rangle$. Setting $\alpha(x,\xi,y,\zeta)=\varphi_1(x,\xi)\varphi_2(y,\zeta)$.
By the same method as Proposition 13 in \cite{D-V-1}, using (\ref{e-80}) for $f_1$ and $f_2$, we obtain that $\langle f^+_1(t), \varphi_1\rangle\langle \bar{f}^+_2(t),\varphi_2 \rangle=\langle\langle f^+_1(t)\bar{f}^+_2(t), \alpha \rangle\rangle$
satisfies
\begin{align*}\notag
&E\langle\langle f^+_1(t)\bar{f}^+_2(t), \alpha \rangle\rangle\\ \notag
&= \langle\langle f_{1,0}\bar{f}_{2,0}, \alpha \rangle\rangle
+E\int^t_0\int_{(\mathbb{T}^d)^2}\int_{\mathbb{R}^2}f_1\bar{f}_2\Big(a(\sqrt{\varepsilon}\xi+1)
\cdot\nabla_x+a(\sqrt{\varepsilon}\zeta+1)\cdot\nabla_y\Big) \alpha d\xi d\zeta dxdyds\\ \notag
 & -\frac{1}{2}E\int^t_0\int_{(\mathbb{T}^d)^2}\int_{\mathbb{R}^2}f_1(s,x,\xi)\partial_{\zeta}\alpha G^2(y,\sqrt{\varepsilon}\zeta+1)d\xi d\nu^{2,\varepsilon,\eta}_{y,s}(\zeta)dxdyds\\ \notag
 & +\frac{1}{2}E\int^t_0\int_{(\mathbb{T}^d)^2}\int_{\mathbb{R}^2}\bar{f}_2(s,y,\zeta)\partial_{\xi}\alpha G^2(x,\sqrt{\varepsilon}\xi+1)d\zeta d\nu^{1,\varepsilon}_{x,s}(\xi)dxdyds\\ \notag
  & -E \int^t_0\int_{(\mathbb{T}^d)^2}\int_{\mathbb{R}^2}G_{1,2}(x,y,\sqrt{\varepsilon}\xi+1,\sqrt{\varepsilon}\zeta+1)\alpha d\nu^{1,\varepsilon}_{x,s}\otimes d\nu^{2,\varepsilon,\eta}_{y,s}(\xi,\zeta)dxdyds\\ \notag
 & +E\int_{(0,t]}\int_{(\mathbb{T}^d)^2}\int_{\mathbb{R}^2}f^{-}_1(s,x,\xi)\partial_{\zeta} \alpha dm^{\varepsilon,\eta}_2(y,\zeta,s)d\xi dx\\ \notag
 & -E\int_{(0,t]}\int_{(\mathbb{T}^d)^2}\int_{\mathbb{R}^2}\bar{f}^+_2(s,y,\zeta)\partial_{\xi} \alpha dm^{\varepsilon}_1(x,\xi,s)d\zeta dy
  +\eta E\int^t_0\int_{(\mathbb{T}^d)^2}\int_{\mathbb{R}^2}f_1\bar{f}_2\Delta_y \alpha d\xi d\zeta dxdyds\\
&=: \langle\langle f_{1,0}\bar{f}_{2,0}, \alpha \rangle\rangle+\sum^7_{i=1}I_i(t),
\end{align*}
where $G^2(x,\xi)=\sum_{k\geq 1}|g_k(x,\xi)|^2$ and {\color{black}$G_{1,2}(x,y,\xi,\zeta)=\sum_{k\geq 1}g_k(x,\xi)g_k(y,\zeta)$}.
Similarly, we have
\begin{align}\notag
&E\langle\langle \bar{f}^+_1(t)f^+_2(t), \alpha \rangle\rangle\\ \notag
&= \langle\langle \bar{f}_{1,0}f_{2,0}, \alpha \rangle\rangle
+E\int^t_0\int_{(\mathbb{T}^d)^2}\int_{\mathbb{R}^2}\bar{f}_1{f}_2\Big(a(\sqrt{\varepsilon}\xi+1)\cdot\nabla_x+a(\sqrt{\varepsilon}\zeta+1)\cdot\nabla_y\Big) \alpha d\xi d\zeta dxdyds\\ \notag
 & +\frac{1}{2}E\int^t_0\int_{(\mathbb{T}^d)^2}\int_{\mathbb{R}^2}\bar{f}_1(s,x,\xi)\partial_{\zeta}\alpha G^2(y,\sqrt{\varepsilon}\zeta+1)d\xi d\nu^{2,\varepsilon,\eta}_{y,s}(\zeta)dxdyds\\ \notag
 & -\frac{1}{2}E\int^t_0\int_{(\mathbb{T}^d)^2}\int_{\mathbb{R}^2}{f}_2(s,y,\zeta)\partial_{\xi}\alpha G^2(x,\sqrt{\varepsilon}\xi+1)d\zeta d\nu^{1,\varepsilon}_{x,s}(\xi)dxdyds\\ \notag
  & + E\int^t_0\int_{(\mathbb{T}^d)^2}\int_{\mathbb{R}^2}G_{1,2}(x,y,\sqrt{\varepsilon}\xi+1,\sqrt{\varepsilon}\zeta+1)\alpha d\nu^{1,\varepsilon}_{x,s}\otimes d\nu^{2,\varepsilon,\eta}_{y,s}(\xi,\zeta)dxdyds\\ \notag
 & -E\int_{(0,t]}\int_{(\mathbb{T}^d)^2}\int_{\mathbb{R}^2}\bar{f}^+_1(s,x,\xi)\partial_{\zeta} \alpha dm^{\varepsilon,\eta}_2(y,\zeta,s)d\xi dx\\ \notag
 & +E\int_{(0,t]}\int_{(\mathbb{T}^d)^2}\int_{\mathbb{R}^2}{f}^{-}_2(s,y,\zeta)\partial_{\xi} \alpha dm^{\varepsilon}_1(x,\xi,s)d\zeta dy
  +\eta E\int^t_0\int_{(\mathbb{T}^d)^2}\int_{\mathbb{R}^2}\bar{f}_1 f_2\Delta_y \alpha d\xi d\zeta dxdyds\\
\label{P-10-1}
&=: \langle\langle \bar{f}_{1,0}{f}_{2,0}, \alpha \rangle\rangle+\sum^7_{i=1}\bar{I}_i(t).
\end{align}
Following the idea developed by \cite{D-V-1}, we can relax the conditions imposed on $\alpha$. Specifically,
we can take $\alpha \in C^{\infty}_b(\mathbb{T}^d_x\times \mathbb{R}_\xi\times \mathbb{T}^d_y\times \mathbb{R}_\zeta)$, which is compactly supported in a neighbourhood of the diagonal
\begin{equation*}
\Big\{(x,\xi,x,\xi); x\in \mathbb{T}^d, \xi\in \mathbb{R}\Big\}.
\end{equation*}
Taking $\alpha=\rho(x-y)\psi(\xi-\zeta)$, then we have the following remarkable identities
\begin{eqnarray}\label{P-11}
(\nabla_x+\nabla_y)\alpha=0, \quad (\partial_{\xi}+\partial_{\zeta})\alpha=0.
\end{eqnarray}
Referring to Proposition 13 in \cite{D-V-1}, we know that $I_5+I_6\leq 0$ and $\bar{I}_5+\bar{I}_6\leq 0$. In view of (\ref{P-11}), we deduce that
\begin{eqnarray*}
  I_1=E\int^t_0\int_{(\mathbb{T}^d)^2}\int_{\mathbb{R}^2}f_1\bar{f}_2\Big(a(\sqrt{\varepsilon}\xi+1)-a(\sqrt{\varepsilon}\zeta+1)\Big) \cdot \nabla_x  \alpha d\xi d\zeta dxdyds.
  \end{eqnarray*}
  and
  \begin{eqnarray*}
  \tilde{I}_1= E\int^t_0\int_{(\mathbb{T}^d)^2}\int_{\mathbb{R}^2}\bar{f}_1f_2\Big(a(\sqrt{\varepsilon}\xi+1)-a(\sqrt{\varepsilon}\zeta+1)\Big)\cdot\nabla_x  \alpha d\xi d\zeta dxdyds.
\end{eqnarray*}
Moreover, {\color{black}using the same method as in the proof of Proposition 13 in \cite{D-V-1}, we have}
\begin{align*}
  \sum^4_{i=2}I_i&=\sum^4_{i=2}\bar{I}_i\\
  &=\frac{1}{2}E\int^t_0\int_{(\mathbb{T}^d)^2}\int_{\mathbb{R}^2}\alpha \sum_{k\geq 1}|g_k(x,\sqrt{\varepsilon}\xi+1)-g_k(y,\sqrt{\varepsilon}\zeta+1)|^2 d\nu^{1,\varepsilon}_{x,s}\otimes\nu^{2,\varepsilon,\eta}_{y,s}(\xi,\zeta)dxdyds.
\end{align*}
Combining all the previous estimates, it follows that
\begin{align}\notag
&E\int_{(\mathbb{T}^d)^2}\int_{\mathbb{R}^2}\rho (x-y)\psi(\xi-\zeta)(f^+_1(x,t,\xi)\bar{f}^+_2(y,t,\zeta)+\bar{f}^+_1(x,t,\xi)f^+_2(y,t,\zeta))d\xi d\zeta dxdy\\
\notag
&\leq  \int_{(\mathbb{T}^d)^2}\int_{\mathbb{R}^2}\rho (x-y)\psi(\xi-\zeta)(f_{1,0}(x,\xi)\bar{f}_{2,0}(y,\zeta)+\bar{f}_{1,0}(x,\xi)f_{2,0}(y,\zeta))d\xi d\zeta dxdy\\ \notag
& +E\int^t_0\int_{(\mathbb{T}^d)^2}\int_{\mathbb{R}^2}(f_1\bar{f}_2+\bar{f}_1f_2)\Big(a(\sqrt{\varepsilon}\xi+1)-a(\sqrt{\varepsilon}\zeta+1)\Big) \cdot\nabla_x  \alpha d\xi d\zeta dxdyds\\ \notag
& +E\int^t_0\int_{(\mathbb{T}^d)^2}\int_{\mathbb{R}^2}\alpha \sum_{k\geq 1}|g_k(x,\sqrt{\varepsilon}\xi+1)-g_k(y,\sqrt{\varepsilon}\zeta+1)|^2 d\nu^{1,\varepsilon}_{x,s}\otimes\nu^{2,\varepsilon,\eta}_{y,s}(\xi,\zeta)dxdyds\\ \notag
& +\eta E\int^t_0\int_{(\mathbb{T}^d)^2}\int_{\mathbb{R}^2}(f_1\bar{f}_2+\bar{f}_1 f_2)\Delta_y \alpha d\xi d\zeta dxdyds\\ \notag
&=: \int_{(\mathbb{T}^d)^2}\int_{\mathbb{R}^2}\rho (x-y)\psi(\xi-\zeta)(f_{1,0}(x,\xi)\bar{f}_{2,0}(y,\zeta)+\bar{f}_{1,0}(x,\xi)f_{2,0}(y,\zeta))d\xi d\zeta dxdy\\
\label{e-9}
& +K_1+K_2+K_3.
\end{align}
Now, taking a sequence $t_n\uparrow t$, since (\ref{e-9}) holds for $f^+(t_n)$. Letting $n\rightarrow \infty$, we deduce that  (\ref{e-9}) holds  for $f^-_i(t)$. Thus, we have
\begin{align}\notag
&E\int_{(\mathbb{T}^d)^2}\int_{\mathbb{R}^2}\rho (x-y)\psi(\xi-\zeta)(f^{\pm}_1(x,t,\xi)\bar{f}^{\pm}_2(y,t,\zeta)+\bar{f}^{\pm}_1(x,t,\xi)f^{\pm}_2(y,t,\zeta))d\xi d\zeta dxdy\\
\notag
&\leq  \int_{(\mathbb{T}^d)^2}\int_{\mathbb{R}^2}\rho (x-y)\psi(\xi-\zeta)(f_{1,0}(x,\xi)\bar{f}_{2,0}(y,\zeta)+\bar{f}_{1,0}(x,\xi)f_{2,0}(y,\zeta))d\xi d\zeta dxdy\\ \label{e-10}
& +K_1+K_2+K_3.
\end{align}
Let $\rho_{\gamma}, \psi_{\delta}$ be approximations to the identity on $\mathbb{T}^d$ and $\mathbb{R}$, respectively. That is, let $\rho\in C^{\infty}(\mathbb{T}^d)$, $\psi\in C^{\infty}_c(\mathbb{R})$ be symmetric non-negative functions such as $\int_{\mathbb{T}^d}\rho =1$, $\int_{\mathbb{R}}\psi =1$ and supp$\psi \subset (-1,1)$. We define
\begin{equation*}
\rho_{\gamma}(x)=\frac{1}{\gamma}\rho\Big(\frac{x}{\gamma}\Big), \quad \psi_{\delta}(\xi)=\frac{1}{\delta}\psi\Big(\frac{\xi}{\delta}\Big).
\end{equation*}
Letting $\rho:=\rho_{\gamma}(x-y)$ and $\psi:=\psi_{\delta}(\xi-\zeta)$ in (\ref{e-10}), we deduce that
\begin{align*}
&E\int_{(\mathbb{T}^d)^2}\int_{\mathbb{R}^2}\rho_{\gamma} (x-y)\psi_{\delta}(\xi-\zeta)(f^{\pm}_1(x,t,\xi)\bar{f}^{\pm}_2(y,t,\zeta)+\bar{f}^{\pm}_1(x,t,\xi)f^{\pm}_2(y,t,\zeta))d\xi d\zeta dxdy\\
&\leq  \int_{(\mathbb{T}^d)^2}\int_{\mathbb{R}^2}\rho_{\gamma} (x-y)\psi_{\delta}(\xi-\zeta)(f_{1,0}(x,\xi)\bar{f}_{2,0}(y,\zeta)+\bar{f}_{1,0}(x,\xi)f_{2,0}(y,\zeta))d\xi d\zeta dxdy\\
&  +\tilde{K}_1(t)+\tilde{K}_2(t)+\tilde{K}_3(t),
\end{align*}
where $\tilde{K}_1, \tilde{K}_2, \tilde{K}_3$ are the corresponding $K_1,K_2,K_3$ in (\ref{e-10}) with $\rho$, $\psi$ replaced by $\rho_{\gamma}$, $\psi_{\delta}$, respectively.

For any $t\in [0,T]$, define the error term
\begin{align}\notag
\mathcal{E}_t(\gamma,\delta)
:=& E\int_{\mathbb{T}^d}\int_{\mathbb{R}}(f^{\pm}_1(x,t,\xi)\bar{f}^{\pm}_2(x,t,\xi)+\bar{f}^{\pm}_1(x,t,\xi)f^{\pm}_2(x,t,\xi))d\xi dx\\
\notag
&-E\int_{(\mathbb{T}^d)^2}\int_{\mathbb{R}^2}(f^{\pm}_1(x,t,\xi)\bar{f}^{\pm}_2(y,t,\zeta)+\bar{f}^{\pm}_1(x,t,\xi){f}^{\pm}_2(y,t,\zeta))\rho_{\gamma}(x-y)\psi_{\delta}(\xi-\zeta)dxdyd\xi d\zeta\\ \notag
&= E\int_{\mathbb{T}^d}\int_{\mathbb{R}}(f^{\pm}_1(x,t,\xi)\bar{f}^{\pm}_2(x,t,\xi)+\bar{f}^{\pm}_1(x,t,\xi)f^{\pm}_2(x,t,\xi))d\xi dx\\ \notag
& -E\int_{(\mathbb{T}^d)^2}\int_{\mathbb{R}}\rho_{\gamma}(x-y)(f^{\pm}_1(x,t,\xi)\bar{f}^{\pm}_2(y,t,\xi)+\bar{f}^{\pm}_1(x,t,\xi)f^{\pm}_2(y,t,\xi))d\xi dxdy\\ \notag
& +E\int_{(\mathbb{T}^d)^2}\int_{\mathbb{R}}\rho_{\gamma}(x-y)(f^{\pm}_1(x,t,\xi)\bar{f}^{\pm}_2(y,t,\xi)+\bar{f}^{\pm}_1(x,t,\xi)f^{\pm}_2(y,t,\xi))d\xi dxdy\\ \notag
& -E\int_{(\mathbb{T}^d)^2}\int_{\mathbb{R}^2}(f^{\pm}_1(x,t,\xi)\bar{f}^{\pm}_2(y,t,\zeta)+\bar{f}^{\pm}_1(x,t,\xi){f}^{\pm}_2(y,t,\zeta))\rho_{\gamma}(x-y)\psi_{\delta}(\xi-\zeta)dxdyd\xi d\zeta\\
\label{qq-3}
&=: H_1+H_2.
\end{align}
By utilizing $\int_{\mathbb{R}}\psi_{\delta}(\xi-\zeta)d\zeta=1$, $\int^{\xi}_{\xi-\delta}\psi_{\delta}(\xi-\zeta)d\zeta=\frac{1}{2}$ and $\int_{(\mathbb{T}^d)^2}\rho_{\gamma}(x-y)dxdy=1$, we deduce that for any $t\in [0,T]$,
\begin{align}\notag
&E\Big|\int_{(\mathbb{T}^d)^2}\int_{\mathbb{R}}\rho_{\gamma}(x-y)f^{\pm}_1(x,t,\xi)\bar{f}^{\pm}_2(y,t,\xi)d\xi dxdy\\ \notag
&-\int_{(\mathbb{T}^d)^2}\int_{\mathbb{R}^2}f^{\pm}_1(x,t,\xi)\bar{f}^{\pm}_2(y,t,\zeta)\rho_{\gamma}(x-y)\psi_{\delta}(\xi-\zeta)dxdyd\xi d\zeta\Big|\\ \notag
&=E\Big|\int_{(\mathbb{T}^d)^2}\rho_{\gamma}(x-y)\int_{\mathbb{R}}I_{v^{\varepsilon, \pm}(x,t)>\xi}\int_{\mathbb{R}}\psi_{\delta}(\xi-\zeta)(I_{v^{\varepsilon,\eta,\pm}(y,t)\leq\xi}-I_{v^{\varepsilon,\eta,\pm}(y,t)\leq \zeta})d\zeta d\xi dxdy\Big|\\ \notag
&\leq E\int_{(\mathbb{T}^d)^2}\int_{\mathbb{R}}\rho_{\gamma}(x-y)I_{v^{\varepsilon,\pm}(x,t)>\xi}\int^{\xi}_{\xi-\delta}\psi_{\delta}(\xi-\zeta)I_{\zeta<v^{\varepsilon,\eta,\pm}(y,t)\leq\xi} d\zeta d\xi dxdy\\ \notag
& +E\int_{(\mathbb{T}^d)^2}\int_{\mathbb{R}}\rho_{\gamma}(x-y)I_{v^{\varepsilon,\pm}(x,t)>\xi}\int^{\xi+\delta}_{\xi}\psi_{\delta}(\xi-\zeta)I_{\xi<v^{\varepsilon,\eta,\pm}(y,t)\leq\zeta} d\zeta d\xi dxdy\\ \notag
&\leq \frac{1}{2}E\int_{(\mathbb{T}^d)^2}\rho_{\gamma}(x-y)\int^{min\{v^{\varepsilon,\pm}(x,t),v^{\varepsilon,\eta,\pm}(y,t)+\delta\}}_{v^{\varepsilon,\eta,\pm}(y,t)}d\xi dxdy +\frac{1}{2}E\int_{(\mathbb{T}^d)^2}\rho_{\gamma}(x-y)\int^{min\{v^{\varepsilon,\pm}(x,s),v^{\varepsilon,\eta,\pm}(y,s)\}}_{v^{\varepsilon,\eta,\pm}(y,s)-\delta}d\xi dxdy\\
%&=& \frac{\delta}{2}E\int_{(\mathbb{T}^d)^2}\rho_{\gamma}(x-y) I_{{ u^{\varepsilon,\pm}(x,t)>u^{\varepsilon,\eta,\pm}(y,t)+\delta }}dxdy\\ \notag
%&&+\frac{1}{2}E\int_{(\mathbb{T}^d)^2}\rho_{\gamma}(x-y)I_{{u^{\varepsilon,\eta,\pm}(y,t)< u^{\varepsilon,\pm}(x,t)\leq u^{\varepsilon,\eta,\pm}(y,t)+\delta }}(u^{\varepsilon,\pm}(x,t)-u^{\varepsilon,\eta,\pm}(y,t) dxdy\\ \notag
%&&+\frac{\delta}{2}E\int_{(\mathbb{T}^d)^2}\rho_{\gamma}(x-y)I_{{u^{\varepsilon,\eta,\pm}(y,t)<u^{\varepsilon,\pm}(x,t)}}dxdy\\ \notag
%&&+\frac{1}{2}E\int_{(\mathbb{T}^d)^2}\rho_{\gamma}(x-y)I_{{u^{\varepsilon,\eta,\pm}(y,t)-\delta<u^{\varepsilon,\pm}(x,t)\leq u^{\varepsilon,\eta,\pm}(y,t)}}(u^{\varepsilon,\pm}(x,t)-u^{\varepsilon,\eta,\pm}(y,t)+\delta) dxdy\\
\label{e-11}
&\leq  \delta.
\end{align}
Similarly,
\begin{align}\notag
&E\Big|\int_{(\mathbb{T}^d)^2}\int_{\mathbb{R}}\rho_{\gamma}(x-y)\bar{f}^{\pm}_1(x,t,\xi){f}^{\pm}_2(y,t,\xi)d\xi dxdy\\
\label{e-12}
&-\int_{(\mathbb{T}^d)^2}\int_{\mathbb{R}^2}\bar{f}^{\pm}_1(x,t,\xi){f}^{\pm}_2(y,t,\zeta)\rho_{\gamma}(x-y)\psi_{\delta}(\xi-\zeta)dxdyd\xi d\zeta\Big|
\leq  \delta.
\end{align}
 (\ref{e-11}) together with (\ref{e-12}) imply that $H_1\leq 2\delta$.

Moreover, when $\gamma$ is small enough, it follows that
\begin{eqnarray*}\notag
&&E\Big|\int_{(\mathbb{T}^d)^2}\int_{\mathbb{R}}\rho_{\gamma}(x-y)f^{\pm}_1(x,t,\xi)\bar{f}^{\pm}_2(y,t,\xi)d\xi dydx-\int_{\mathbb{T}^d}\int_{\mathbb{R}}f^{\pm}_1(x,t,\xi)\bar{f}^{\pm}_2(x,t,\xi)d\xi dx\Big|\\ \notag
&=&E\Big|\int_{(\mathbb{T}^d)^2}\int_{\mathbb{R}}\rho_{\gamma}(x-y)f^{\pm}_1(x,t,\xi)\bar{f}^{\pm}_2(y,t,\xi)d\xi dydx-\int_{\mathbb{T}^d}\int_{|z|<\gamma}\int_{\mathbb{R}}\rho_{\gamma}(z)f^{\pm}_1(x,t,\xi)\bar{f}^{\pm}_2(x,t,\xi)d\xi dzdx\Big|\\ \notag
&=&E\Big|\int_{(\mathbb{T}^d)^2}\int_{\mathbb{R}}\rho_{\gamma}(x-y)f^{\pm}_1(x,t,\xi)(\bar{f}^{\pm}_2(y,t,\xi)-\bar{f}^{\pm}_2(x,t,\xi))d\xi dydx\Big|\\ \notag
&\leq&\sup_{|z|<\gamma}E\int_{\mathbb{T}^d}\int_{\mathbb{R}}f^{\pm}_1(x,t,\xi)|\bar{f}^{\pm}_2(x-z,t,\xi)-\bar{f}^{\pm}_2(x,t,\xi)|d\xi dx\\ \label{rrr-1}
&\leq& %\sup_{|z|<\gamma}\int_{\mathbb{T}^d}\int_{\mathbb{R}}|-f^{\pm}_2(x-z,t,\xi)+I_{0>\xi}-I_{0>\xi}+f^{\pm}_2(x,t,\xi)|d\xi dx\\
 \sup_{|z|<\gamma}E\int_{\mathbb{T}^d}\int_{\mathbb{R}}|\Lambda_{f^{\pm}_2}(x-z,t,\xi)-\Lambda_{f^{\pm}_2}(x,t,\xi)|d\xi dx,
\end{eqnarray*}
where $\Lambda_{f_2}(\cdot,\cdot,\xi):=f_2(\cdot,\cdot,\xi)-I_{0>\xi}$. From (\ref{r-5}), we know that $\Lambda_{f_2}$ is integrable in $L^1(\Omega\times \mathbb{T}^d\times \mathbb{R})$ uniformly with respect to $\varepsilon\in (0,1)$ and $\eta>0$, it follows that
\begin{align}\label{qq-1}
\lim_{\gamma\rightarrow 0}\sup_{\varepsilon\in(0,1),\eta>0}E\Big|\int_{(\mathbb{T}^d)^2}\int_{\mathbb{R}}\rho_{\gamma}(x-y)f^{\pm}_1(x,t,\xi)\bar{f}^{\pm}_2(y,t,\xi)d\xi dydx-\int_{\mathbb{T}^d}\int_{\mathbb{R}}f^{\pm}_1(x,t,\xi)\bar{f}^{\pm}_2(x,t,\xi)d\xi dx\Big|=0.
\end{align}
Similarly, it holds that
\begin{align}\label{qq-2}
\lim_{\gamma\rightarrow 0}\sup_{\varepsilon\in(0,1),\eta>0}E\Big|\int_{(\mathbb{T}^d)^2}\int_{\mathbb{R}}\rho_{\gamma}(x-y)\bar{f}^{\pm}_1(x,t,\xi)f^{\pm}_2(y,t,\xi)d\xi dxdy-\int_{\mathbb{T}^d}\int_{\mathbb{R}}\bar{f}^{\pm}_1(x,t,\xi)f^{\pm}_2(x,t,\xi)d\xi dx\Big|=0.
\end{align}
%Moreover, it follows that
%\begin{eqnarray}\notag
%&&\Big|E\int_{(\mathbb{T}^d)^2}\int_{\mathbb{R}}\rho_{\gamma}(x-y)f^{\pm}_1(x,t,\xi)\bar{f}^{\pm}_2(y,t,\xi)d\xi dxdy-\int_{\mathbb{T}^d}\int_{\mathbb{R}}f^{\pm}_1(x,t,\xi)\bar{f}^{\pm}_2(x,t,\xi)d\xi dx\Big|\\ \notag
%&=& \Big|E\int_{(\mathbb{T}^d)^2}\rho_{\gamma}(x-y)\int_{\mathbb{R}}I_{v^{\varepsilon,\pm}(x,t)>\xi}
%[I_{v^{\varepsilon,\eta,\pm}(x,t)\leq\xi}-I_{v^{\varepsilon,\eta,\pm}(y,t)\leq\xi}]d\xi dxdy\Big|\\ \notag
%&=& \Big|E\int_{(\mathbb{T}^d)^2}\rho_{\gamma}(x-y)\Big(v^{\varepsilon,\eta,\pm}(y,t)-v^{\varepsilon,\eta,\pm}(x,t)\Big)dxdy\Big|\\
%\label{e-13}
%&\rightarrow& 0, as\ \gamma\rightarrow0,
%\end{eqnarray}
%uniformly of $\varepsilon,\eta\in[0,1]$. Similarly, it holds that
%\begin{align*}\label{qq-2}
%&\Big|E\int_{(\mathbb{T}^d)^2}\int_{\mathbb{R}}\rho_{\gamma}(x-y)\bar{f}^{\pm}_1(x,t,\xi)f^{\pm}_2(y,t,\xi)d\xi dxdy-\int_{\mathbb{T}^d}\int_{\mathbb{R}}\bar{f}^{\pm}_1(x,t,\xi)f^{\pm}_2(x,t,\xi)d\xi dx\Big|\\
%\rightarrow& 0, as\ \gamma\rightarrow0,
%\end{align*}
%uniformly of $\varepsilon,\eta\in[0,1]$.
Therefore, we conclude that
for any $t\in [0,T]$,
\begin{align}\label{qq-5}
\lim_{\gamma,\delta\rightarrow0}\sup_{\varepsilon\in(0,1),\eta>0}|\mathcal{E}_t(\gamma,\delta)|=0.
\end{align}
%In particular,
%\begin{eqnarray}\label{qq-5}
%\lim_{\gamma,\delta\rightarrow0}\sup_{\varepsilon,\eta\in[0,1]}|\mathcal{E}_0(\gamma,\delta)|=0.
%\end{eqnarray}
By the dominate convergence theorem, we have
\begin{align}\label{qq-4-1}
\lim_{\gamma,\delta\rightarrow0}\sup_{\varepsilon\in(0,1),\eta>0}\int_0^T|\mathcal{E}_t(\gamma,\delta)|dt=0.
\end{align}

In the following, we aim to make estimates of $\tilde{K}_1(t)$, $\tilde{K}_2(t)$ and $\tilde{K}_3(t)$.
We begin with the estimation of $\tilde{K}_1(t)$. By Hypothesis H, we deduce that
\begin{align*}
  \tilde{K}_1(t)&\leq \sqrt{\varepsilon} E\int^t_0\int_{(\mathbb{T}^d)^2}\int_{\mathbb{R}^2}(f_1\bar{f}_2+\bar{f}_1f_2)\
  \Gamma(\sqrt{\varepsilon}\xi+1,\sqrt{\varepsilon}\zeta+1)|\xi-\zeta| |\nabla_x \rho_{\gamma}(x-y)|\psi_{\delta}(\xi-\zeta) d\xi d\zeta dxdyds\notag\\
  &= \sqrt{\varepsilon}E\int^t_0\int_{(\mathbb{T}^d)^2}|\nabla_x \rho_{\gamma}(x-y)|\int_{\mathbb{R}^2}l(\xi,\zeta) d\nu^{1,\varepsilon}_{x,s}\otimes d\nu^{2}_{y,s}(\xi,\zeta) dxdyds.\notag
\end{align*}
where
\begin{align*}
l(\xi, \zeta)=\int^{\infty}_{\zeta}\int^{\xi}_{-\infty}\Gamma(\sqrt{\varepsilon}\xi'+1,\sqrt{\varepsilon}\zeta'+1)|\xi'-\zeta'|\psi_{\delta}(\xi'-\zeta')d\xi'd\zeta'.
\end{align*}
Let $\xi''=\xi'-\zeta'$, by Hypothesis H, it follows that
\begin{align*}
l(\xi, \zeta)&\leq \int^{\infty}_{\zeta}\left(\int_{\{|\xi''|<\delta,\xi''<\xi-\zeta' \}}\Gamma(\sqrt{\varepsilon}(\xi''+\zeta')+1,\sqrt{\varepsilon}\zeta'+1)|\xi''|\psi_{\delta}(\xi'')d\xi''\right)d\zeta'\notag\\
&\leq  \delta\int^{\infty}_{\zeta}\left(\int_{\{|\xi''|<\delta,\xi''<\xi-\zeta' \}}(1+|\sqrt{\varepsilon}(\xi''+\zeta')+1|^{p-1}+|\sqrt{\varepsilon}\zeta'+1|^{p-1})\psi_{\delta}(\xi'')d\xi''\right)d\zeta'\notag\\
&\leq \delta \int^{\xi+\delta}_{\zeta}\left((1+|\sqrt{\varepsilon}(\xi+\zeta')+1|^{p-1}+|\sqrt{\varepsilon}\zeta'+1|^{p-1})\right)d\zeta'\notag\\
&\leq \delta \int^{\xi+\delta}_{\zeta}\left(C(p)+|\sqrt{\varepsilon}\delta|^{p-1}+|\sqrt{\varepsilon}\zeta'|^{p-1}+|\sqrt{\varepsilon}\zeta'|^{p-1})\right)d\zeta'\notag\\
&\leq\delta [C(p)+|\sqrt{\varepsilon}\delta|^{p-1}](|\xi|+|\zeta|+\delta)+C(p)\delta(\sqrt{\varepsilon})^{p-1}[|\xi|^p+|\zeta|^p]+C(p)(\sqrt{\varepsilon})^{p-1}\delta^{p+1}\notag\\
&\leq C(p)\delta^2+ C(p)(\sqrt{\varepsilon})^{p-1}\delta^{p+1}+\delta (\sqrt{\varepsilon})^{-1}\Big(C(p)+|\sqrt{\varepsilon}\delta|^{p-1}\Big)\Big(|\sqrt{\varepsilon}\xi|+|\sqrt{\varepsilon}\zeta|\Big)\notag\\
&
+C(p)\delta(\sqrt{\varepsilon})^{-1}\Big(|\sqrt{\varepsilon}\xi|^p+|\sqrt{\varepsilon}\zeta|^p\Big).\notag
%\Big[1+(\sqrt{\varepsilon}\delta)^{p-1}\Big]\Big[|\sqrt{\varepsilon}\xi|^p+|\sqrt{\varepsilon}\zeta|^p\Big].
\end{align*}
As a consequence, we get
 \begin{align*}\notag
\tilde{K}_1(t)
&\leq C(p)\delta^2\sqrt{\varepsilon}\gamma^{-1}
+ C(p)(\sqrt{\varepsilon})^{p}\gamma^{-1}\delta^{p+1}\notag\\
& +\delta\gamma^{-1}\Big(C(p)+|\sqrt{\varepsilon}\delta|^{p-1}\Big)\Big(E\|\sqrt{\varepsilon}v^{\varepsilon}\|_{L^1([0,T];L^1(\mathbb{T}^d))}+E\|\sqrt{\varepsilon}v^{\varepsilon,\eta}\|_{L^1([0,T];L^1(\mathbb{T}^d))}\Big)
\notag\\
& +C(p)\delta\gamma^{-1}\Big(E\|\sqrt{\varepsilon}v^{\varepsilon}\|^p_{L^p([0,T];L^p(\mathbb{T}^d))}+E\|\sqrt{\varepsilon}v^{\varepsilon,\eta}\|^p_{L^p([0,T];L^p(\mathbb{T}^d))}\Big)\notag\\
&\leq C(p)\delta^2\sqrt{\varepsilon}\gamma^{-1}
+ C(p)(\sqrt{\varepsilon})^{p}\gamma^{-1}\delta^{p+1}
+C(T)\delta\gamma^{-1}\Big(C(p)+|\sqrt{\varepsilon}\delta|^{p-1}\Big),\notag
\end{align*}
where we have used (\ref{r-15})-(\ref{r-16}).
By utilizing Hypothesis H, we have
\begin{align*}
\tilde{K}_2(t)&\leq D_1
E\int^t_0\int_{(\mathbb{T}^d)^2}\int_{\mathbb{R}^2}\rho_{\gamma}(x-y)\psi_{\delta}(\xi-\zeta) [|x-y|^2+|\sqrt{\varepsilon}(\xi-\zeta)|^2] d\nu^{1,\varepsilon}_{x,s}\otimes\nu^{2,\varepsilon,\eta}_{y,s}(\xi,\zeta)dxdyds\notag\\
&\leq D_1t\gamma^2+ D_1t|\sqrt{\varepsilon}\delta|^2\delta^{-1}.\notag
\end{align*}
%Proceeding as Proposition 3.2 in \cite{D-V-1}, we have
%\begin{eqnarray*}
%\tilde{K}_2(t)\leq tD_1\gamma^2 \delta^{-1}+ \frac{\varepsilon}{2}TD_1\delta^2.
%\end{eqnarray*}
For the remainder {\color{black}term} $\tilde{K}_3$, it can be estimated as follows
\begin{align*}
 \tilde{K}_3(t)\leq \eta E \int^t_0\int_{(\mathbb{T}^d)^2}\int_{\mathbb{R}^2}(f_1\bar{f}_2+\bar{f}_1f_2)|\Delta_y \rho_{\gamma}(x-y)|\psi_{\delta}(\xi-\zeta) d\xi d\zeta dxdyds\leq  \eta t\gamma^{-2}.
\end{align*}
%where $\nu^{1,\varepsilon}_{x,s}(\xi)=\partial_{\xi}\bar{f}_1(x,s,\xi)$, $\nu^{2,\varepsilon,\eta}_{y,s}(\zeta)=\partial_{\zeta}\bar{f}_2(y,s,\zeta)$, and
%\begin{eqnarray*}
%l(\xi, \zeta)=\int^{\infty}_{\zeta}\int^{\xi}_{-\infty}\psi_{\delta}(\xi'-\zeta')d\xi'd\zeta'.
%\end{eqnarray*}
%Moreover, let $\xi''=\xi'-\zeta'$, it follows that
%\begin{eqnarray*}
%l(\xi, \zeta)&\leq & \int^{\infty}_{\zeta}\left(\int_{\{|\xi''|<\delta,\xi''<\xi-\zeta' \}}\psi_{\delta}(\xi'')d\xi''\right)d\zeta'\\
%&\leq & C\delta\int^{\xi+\delta}_{\zeta} \|\psi_{\delta}\|_{L^{\infty}}d\zeta'\\
%&\leq &C(|\xi|+|\zeta|+\delta).
%\end{eqnarray*}
%Then, using the property that the measures $\nu^{1,\varepsilon}_{x,s}$ and $\nu^{2,\varepsilon,\eta}_{y,s}$ vanish at the infinity, it yields
%\begin{eqnarray*}
% \tilde{K}_3\leq C(1+\delta)\eta T\gamma^{-2}.
%\end{eqnarray*}
Based on the above estimates, we  get
\begin{align*}
&E\int_{\mathbb{T}^d}\int_{\mathbb{R}}(f^{\pm}_1(x,t,\xi)\bar{f}^{\pm}_2(x,t,\xi)+\bar{f}^{\pm}_1(x,t,\xi)f^{\pm}_2(x,t,\xi))d\xi dx\notag\\
&=E\int_{(\mathbb{T}^d)^2}\int_{\mathbb{R}^2}\rho_{\gamma} (x-y)\psi_{\delta}(\xi-\zeta)(f^{\pm}_1(x,t,\xi)\bar{f}^{\pm}_2(y,t,\zeta)+\bar{f}^{\pm}_1(x,t,\xi)f^{\pm}_2(y,t,\zeta))d\xi d\zeta dxdy+\mathcal{E}_t(\gamma,\delta)\notag\\
&\leq  \int_{\mathbb{T}^d}\int_{\mathbb{R}}(f_{1,0}(x,\xi)\bar{f}_{2,0}(x,\xi)+\bar{f}_{1,0}(x,\xi)f_{2,0}(x,\xi))d\xi dx+|\mathcal{E}_0(\gamma,\delta)|+\mathcal{E}_t(\gamma,\delta)\notag\\
& +C(p)\delta^2\sqrt{\varepsilon}\gamma^{-1}
+ C(p)(\sqrt{\varepsilon})^{p}\gamma^{-1}\delta^{p+1}
+C(T)\delta\gamma^{-1}\Big(C(p)+|\sqrt{\varepsilon}\delta|^{p-1}\Big)\notag\\
&+D_1t\gamma^2+ D_1t|\sqrt{\varepsilon}\delta|^2\delta^{-1}+\eta t\gamma^{-2}.\notag
\end{align*}
Then
\begin{align*}
&E\int^T_0\int_{\mathbb{T}^d}\int_{\mathbb{R}}(f^{\pm}_1(x,t,\xi)\bar{f}^{\pm}_2(x,t,\xi)+\bar{f}^{\pm}_1(x,t,\xi)f^{\pm}_2(x,t,\xi))d\xi dxdt\notag\\
&\leq  T\int_{\mathbb{T}^d}\int_{\mathbb{R}}(f_{1,0}(x,\xi)\bar{f}_{2,0}(x,\xi)+\bar{f}_{1,0}(x,\xi)f_{2,0}(x,\xi))d\xi dx+T|\mathcal{E}_0(\gamma,\delta)|+\int^T_0\mathcal{E}_t(\gamma,\delta)dt\notag\\
& + C(p)\delta^2\sqrt{\varepsilon}\gamma^{-1}
+ C(p)(\sqrt{\varepsilon})^{p}\gamma^{-1}\delta^{p+1}
+C(T)\delta\gamma^{-1}\Big(C(p)+|\sqrt{\varepsilon}\delta|^{p-1}\Big)\notag\\
& +D_1T^2\gamma^2+ D_1T^2|\sqrt{\varepsilon}\delta|^2\delta^{-1}+T^2\eta\gamma^{-2}.\notag
\end{align*}
%Taking $\delta=\gamma^{\frac{4}{3}}$ and $\gamma=\eta^{\frac{1}{3}}$, we have
%\begin{eqnarray*}
%&&E\int^T_0\int_{\mathbb{T}^d}\int_{\mathbb{R}}(f^{\pm}_1(x,t,\xi)\bar{f}^{\pm}_2(x,t,\xi)+\bar{f}^{\pm}_1(x,t,\xi)f^{\pm}_2(x,t,\xi))d\xi dxdt\\
%&\leq & T\int_{\mathbb{T}^d}\int_{\mathbb{R}}(f_{1,0}(x,\xi)\bar{f}_{2,0}(x,\xi)+\bar{f}_{1,0}(x,\xi)f_{2,0}(x,\xi))d\xi dx+T|\mathcal{E}^{\varepsilon}_0(\gamma,\delta)|+\Big|\int^T_0\mathcal{E}^{\varepsilon}_t(\gamma,\delta)dt\Big|\\
%&& +\sqrt{\varepsilon}T^2C(q) \eta^{\frac{1}{9}} (1+\varepsilon^{\frac{p-1}{2}}C(q))+T^2D_1\eta^{\frac{2}{9}}+T^2D_1\eta^{\frac{4}{9}}+C(1+\eta^{\frac{4}{9}})\eta^{\frac{1}{3}} T^2.
%\end{eqnarray*}
Utilizing the following identities
\begin{align}\label{rrr-8}
\int_{\mathbb{R}}I_{v^{\varepsilon}>\xi}\overline{I_{v^{\varepsilon,\eta}>\xi}}d\xi=(v^{\varepsilon}-v^{\varepsilon,\eta})^+,\quad
\int_{\mathbb{R}}\overline{I_{v^{\varepsilon}>\xi}}I_{v^{\varepsilon,\eta}>\xi}d\xi=(v^{\varepsilon}-v^{\varepsilon,\eta})^-,
\end{align}
we get
\begin{align*}
&E\|v^{\varepsilon}-v^{\varepsilon,\eta}\|_{L^1([0,T];L^1(\mathbb{T}^d))}\notag\\
&\leq T|\mathcal{E}_0(\gamma,\delta)|+\int^T_0\mathcal{E}_t(\gamma,\delta)dt
+C(p)\delta^2\sqrt{\varepsilon}\gamma^{-1}
+ C(p)(\sqrt{\varepsilon})^{p}\gamma^{-1}\delta^{p+1}
\notag\\
&+C(T)\delta\gamma^{-1}\Big(C(p)+|\sqrt{\varepsilon}\delta|^{p-1}\Big)+D_1T^2\gamma^2+ D_1T^2|\sqrt{\varepsilon}\delta|^2\delta^{-1}+T^2\eta\gamma^{-2}.\notag
\end{align*}
%Letting $\eta\rightarrow0$, we have
%\begin{eqnarray*}
%&&E\|Y^{\varepsilon}-Y^{\varepsilon,\eta}\|_{L^1([0,T];L^1(\mathbb{T}^d))}\\
%&\leq & \sqrt{\varepsilon}T^2C(q) \eta^{\frac{1}{9}} (1+\varepsilon^{\frac{p-1}{2}}C(q))+T^2D_1\eta^{\frac{2}{9}}+T^2D_1\eta^{\frac{4}{9}}+C(1+\eta^{\frac{4}{9}})\eta^{\frac{1}{3}} T^2 .
%\end{eqnarray*}

Recall $p\geq 1$, taking $\delta=\gamma^{\frac{4}{3}}$, $\gamma=\eta^{\frac{1}{3}}$, we have
\begin{align*}
&\sup_{\varepsilon\in(0,1)} E\Big\|\frac{u^{\varepsilon}-\bar{u}}{\sqrt{\varepsilon}}-\frac{u^{\varepsilon,\eta}-\bar{u}^{\eta}}{\sqrt{\varepsilon}}\Big\|_{L^1([0,T];L^1(\mathbb{T}^d))}\\
&\leq  T|\mathcal{E}_0(\gamma,\delta)|+\int^T_0\mathcal{E}_t(\gamma,\delta)dt
+C(p,T)\eta^{\frac{1}{9}}
+D_1T^2\eta^{\frac{2}{3}}+ D_1T^2\eta^{\frac{4}{9}}+T^2\eta^{\frac{1}{3}}.
\end{align*}
Thus, by (\ref{qq-5}) and (\ref{qq-4-1}), we have
\begin{eqnarray*}
\lim_{\eta\rightarrow 0}\sup_{\varepsilon\in(0,1)}E\Big\|\frac{u^{\varepsilon}-\bar{u}}{\sqrt{\varepsilon}}-\frac{u^{\varepsilon,\eta}-\bar{u}^{\eta}}{\sqrt{\varepsilon}}\Big\|_{L^1([0,T];L^1(\mathbb{T}^d))}= 0,
\end{eqnarray*}
%Let $\Phi\equiv 0$, we have
%There exists $\eta=\varepsilon^2$ such that
%\begin{eqnarray*}
%\lim_{\varepsilon\rightarrow 0}\sup_{n>0}\Big\|\frac{\bar{u}^{\eta}-u^{0}}{\sqrt{\varepsilon}}\Big\|_{L^1([0,T];L^1(\mathbb{T}^d))}=0.
%\end{eqnarray*}
which is the desiblack result.
\end{proof}

%
%Let $v^{0,\eta}$ be the solution of the following SPDE:
%\begin{eqnarray}
%  \left\{
%    \begin{array}{ll}
%      dv^{0,\eta}-\eta \Delta v^{0,\eta}dt+\nabla (a(u^{0,\eta})v^{0,\eta})dt=\Phi(u^{0,\eta})dW(t), \\
%      v^{0,\eta}(0)=0,
%    \end{array}
%  \right.
%\end{eqnarray}

Now, we aim to make estimates of the second term of the right hand side of (\ref{3terms}). That is,
\begin{prp}\label{prp-4}
For any fixed $\eta>0$, we have
\begin{equation*}
\lim_{\varepsilon\rightarrow0}\sup_{t\in[0,T]}E\Big\|\frac{u^{\varepsilon,\eta}(t)-\bar{u}^{\eta}(t)}{\sqrt{\varepsilon}}-\bar{u}^{1,\eta}(t)\Big\|_{H}^2=0.
\end{equation*}
\end{prp}

%To achieve the above result, we need a priori estimate.
%\begin{lemma} For the solution $u^{\varepsilon,\eta}$ of (\ref{uepe}), it holds that
%\begin{equation}\label{energy}
%\sup_{\varepsilon\in(0,1]}\Big[\sup_{t\in[0,T]}E\|u^{\varepsilon,\eta}(t)\|_{H}^2+2\eta E\int_0^T\|\nabla u^{\varepsilon,\eta}(s)\|_{H}^2ds\Big]\leq C(T).
%\end{equation}
%\end{lemma}
%
%\begin{proof}
%Taking $L^2$ inner product, using It\^o formula and by Hypothesis H, we get
%\begin{align*}
%\sup_{t\in[0,T]}E\|u^{\varepsilon,\eta}(t)\|_{H}^2+2\eta E\int_0^T\|\nabla u^{\varepsilon,\eta}(s)\|_{H}^2ds\leq&1+\varepsilon E\int_0^T\int_{\mathbb{T}^d}\sum_{k\geq1}|g_k(x,u^{\varepsilon,\eta})|^2dxds\\
%\leq&1+\varepsilon D_0 E\int_0^T(1+\|u^{\varepsilon,\eta}(s)\|_{H}^2)ds.
%\end{align*}
%Then Gronwall's inequality yields the desiblack result.
%\end{proof}
%
%
%Now, we are ready to present the proof of Proposition \ref{prp-4}.
%[\textbf{ Proof of Proposition \ref{prp-4}}]
\begin{proof}
Fix $\eta>0$. Denote by $w_1^{\varepsilon,\eta}:=\frac{u^{\varepsilon,\eta}-\bar{u}^{\eta}}{\sqrt{\varepsilon}}-\bar{u}^{1,\eta}
=v^{\varepsilon,\eta}-\bar{u}^{1,\eta}$, from  (\ref{vepe}) and (\ref{rrr-2}), we get
\begin{equation*}
dw_1^{\varepsilon,\eta}+ {\rm{div}} \left(\frac{A(\sqrt{\varepsilon}v^{\varepsilon,\eta}+1)-A(1)}{\sqrt{\varepsilon}}-a(\bar{u})\bar{u}^{1,\eta}\right)dt
=\eta\Delta w_1^{\varepsilon,\eta}dt+(\Phi(\sqrt{\varepsilon}v^{\varepsilon,\eta}+1)-\Phi(\bar{u}))dW(t),
\end{equation*}
with $w_1^{\varepsilon,\eta}(0)=0$.

Applying It\^o formula and (\ref{equ-29}), we obtain
\begin{align*}
&E\|w_1^{\varepsilon,\eta}(t)\|^2_{H}+2\eta E\int_0^t\|\nabla  w_1^{\varepsilon,\eta}(s)\|_{H}^2ds
= -2E\int^t_0\left(\nabla w_1^{\varepsilon,\eta},\frac{A(\sqrt{\varepsilon}v^{\varepsilon,\eta}+1)-A(1)}{\sqrt{\varepsilon}}-a(\bar{u})\bar{u}^{1,\eta}\right)
\\
&+  \sum_{k\geq 1}E\int^t_0\int_{\mathbb{T}^d}|g_k(x,\sqrt{\varepsilon}v^{\varepsilon,\eta}+1)-g_k(x,\bar{u})|^2dxds
=:  J_1+J_2.
\end{align*}
By Hypothesis H, we reach
\begin{align*} &\left|\frac{A(\sqrt{\varepsilon}v^{\varepsilon,\eta}+1)-A(1)}{\sqrt{\varepsilon}}
-a(\bar{u})\bar{u}^{1,\eta}\right|\\
\leq & (\sqrt{\varepsilon})^p|v^{\varepsilon,\eta}|^{p+1}
+C(p)(\sqrt{\varepsilon})^{p-1}|v^{\varepsilon,\eta}|^{p}+\dots+\sqrt{\varepsilon}|v^{\varepsilon,\eta}|^2
+|a(1)v^{\varepsilon,\eta}-a(1)\bar{u}^{1,\eta}|\\
\leq&  (\sqrt{\varepsilon})^p|v^{\varepsilon,\eta}|^{p+1}
+C(p)(\sqrt{\varepsilon})^{p-1}|v^{\varepsilon,\eta}|^{p}+\dots+a(1)|w_1^{\varepsilon,\eta}|.
\end{align*}
%It follows that
%\begin{align*}
% J_1\leq& 2E\int^t_0 \Big(\nabla w_1^{\varepsilon,\eta},(\sqrt{\varepsilon})^p(v^{\varepsilon,\eta})^{p+1}
%+C(p)(\sqrt{\varepsilon})^{p-1}(v^{\varepsilon,\eta})^{p}+\dots+a(1)w_1^{\varepsilon,\eta}\Big)ds\\
%+ & \sum_{k\geq 1}E\int^t_0\int_{\mathbb{T}^d}|g_k(x,\sqrt{\varepsilon}v^{\varepsilon,\eta}+1)-g_k(x,\bar{u})|^2dxds\\
%=: & J_1+J_2.
%\end{align*}
Then, by using Young inequality, it gives
\begin{align*}
  J_1\leq & \eta E\int_0^t\|\nabla  w_1^{\varepsilon,\eta}(s)\|_{H}^2ds
 +C(p,\eta)\varepsilon^p E\int^t_0 \|v^{\varepsilon,\eta}\|^{2(p+1)}_{L^{2(p+1)}(\mathbb{T}^d)}ds\\
+& C(p,\eta)\varepsilon^{p-1}E\int^t_0 \|v^{\varepsilon,\eta}\|^{2p}_{L^{2p}(\mathbb{T}^d)}ds+\dots+\varepsilon E\int^t_0 \|v^{\varepsilon,\eta}\|^{4}_{L^{4}(\mathbb{T}^d)}ds
+C(p,\eta)a(1)E\int^t_0 \|w_1^{\varepsilon,\eta}\|^2_Hds.
\end{align*}
Owing to (\ref{r-5}), we reach
\begin{align*}
  J_1\leq & \eta E\int_0^t\|\nabla  w_1^{\varepsilon,\eta}(s)\|_{H}^2ds
 +C(p,T,\eta)\Big(\varepsilon^p+\varepsilon^{p-1}+\dots+\varepsilon\Big) \\
+& C(p,\eta)a(1)E\int^t_0 \|w_1^{\varepsilon,\eta}\|^2_Hds.
\end{align*}
By Hypothesis H and  (\ref{r-5}), it gives
\begin{align*}
  J_2\leq D_1\varepsilon E\int^t_0\|v^{\varepsilon,\eta}\|^2_Hds\leq C(D_1,T)\varepsilon.
\end{align*}
As a result, it follows that
\begin{align*}
&E\|w_1^{\varepsilon,\eta}(t)\|^2_{H}+\eta E\int_0^t\|\nabla  w_1^{\varepsilon,\eta}(s)\|_{H}^2ds\\
\leq& C(p,T,\eta)\Big(\varepsilon^p+\varepsilon^{p-1}+\dots+\varepsilon\Big)
+ C(p,\eta)a(1)E\int^t_0 \|w_1^{\varepsilon,\eta}\|^2_Hds
+ C(D_1,T)\varepsilon.
\end{align*}
By Gronwall inequality, we get for any $\eta>0$,
\begin{align*}
&\sup_{t\in [0,T]}E\|w_1^{\varepsilon,\eta}(t)\|^2_{H}\leq C(p,T,\eta,D_1)\varepsilon.
\end{align*}

\end{proof}

Based on all the previous estimates, we are able to proceed with the proof of a central limit theorem. It reads as follows.
\begin{thm}{\rm{(Central Limit Theorem)}}
Assume that Hypothesis H is in force, then
\begin{eqnarray*}
\lim_{\varepsilon\rightarrow 0}E\Big\|\frac{u^{\varepsilon}-\bar{u}}{\sqrt{\varepsilon}}-\bar{u}^1\Big\|_{L^1([0,T];L^1(\mathbb{T}^d))}=0.
\end{eqnarray*}

\end{thm}
\begin{proof}
Recall that for any $\eta>0$, we have
\begin{align}\notag
&E\Big\|\frac{u^{\varepsilon}-\bar{u}}{\sqrt{\varepsilon}}-\bar{u}^1\Big\|_{L^1([0,T];L^1(\mathbb{T}^d))}\\ \notag
\leq &E\Big\|\frac{u^{\varepsilon}-\bar{u}}{\sqrt{\varepsilon}}-\frac{u^{\varepsilon,\eta}-\bar{u}^{\eta}}{\sqrt{\varepsilon}}\Big\|_{L^1([0,T];L^1(\mathbb{T}^d))}
+E\Big\|\frac{u^{\varepsilon,\eta}-\bar{u}^{\eta}}{\sqrt{\varepsilon}}-\bar{u}^{1,\eta}\Big\|_{L^1([0,T];L^1(\mathbb{T}^d))}\\
\label{3terms-1}
+&E\|\bar{u}^{1,\eta}-\bar{u}^1\|_{L^1([0,T];L^1(\mathbb{T}^d))}.
\end{align}
From Proposition \ref{prp-2} and (\ref{r-3}), we know that for any $\delta>0$, we can choose $\eta_0>0$ small enough, such that for all $\varepsilon>0$,
\begin{equation*}
E\Big\|\frac{u^{\varepsilon}-\bar{u}}{\sqrt{\varepsilon}}-\frac{u^{\varepsilon,\eta_0}-\bar{u}^{\eta_0}}{\sqrt{\varepsilon}}\Big\|_{L^1([0,T];L^1(\mathbb{T}^d))}<\frac{\delta}{3},\quad E\|\bar{u}^{1,\eta_0}-\bar{u}^1\|_{L^1([0,T];L^1(\mathbb{T}^d))}<\frac{\delta}{3}.
\end{equation*}
Letting $\eta=\eta_0$, we deduce from (\ref{3terms-1}) that
\begin{align}\label{3terms-2}
E\Big\|\frac{u^{\varepsilon}-\bar{u}}{\sqrt{\varepsilon}}-\bar{u}^1\Big\|_{L^1([0,T];L^1(\mathbb{T}^d))}
\leq \frac{2}{3}\delta+ E\Big\|\frac{u^{\varepsilon,\eta_0}-\bar{u}^{\eta_0}}{\sqrt{\varepsilon}}-\bar{u}^{1,\eta_0}\Big\|_{L^1([0,T];L^1(\mathbb{T}^d))}.
\end{align}
By Proposition \ref{prp-4}, for the above $\delta>0$, there exists small enough $\varepsilon_0>0$ such that for any $\varepsilon\leq \varepsilon_0$,
\begin{equation*}
E\Big\|\frac{u^{\varepsilon,\eta_0}-\bar{u}^{\eta_0}}{\sqrt{\varepsilon}}-\bar{u}^{1,\eta_0}\Big\|_{L^1([0,T];L^1(\mathbb{T}^d))}
<{\color{black}\frac{\delta}{3}T}.
\end{equation*}
Thus, for any $\varepsilon\leq \varepsilon_0$, we deduce from (\ref{3terms-2}) that
\begin{equation*}
E\Big\|\frac{u^{\varepsilon}-\bar{u}}{\sqrt{\varepsilon}}-\bar{u}^1\Big\|_{L^1([0,T];L^1(\mathbb{T}^d))}<\frac{2\delta}{3}+\frac{\delta}{3}T.
\end{equation*}
By the arbitrary of $\delta$, we conclude the result.

%We have already know that
%\begin{align*}
%&\|\frac{u^{\varepsilon}-\bar{u}}{\sqrt{\varepsilon}}-\bar{u}^1\|_{L^1([0,T];L^1(\mathbb{T}^d))}\\
%&\leq\|\frac{u^{\varepsilon}}{\sqrt{\varepsilon}}-\frac{u^{\varepsilon,\eta}}{\sqrt{\varepsilon}}\|_{L^1([0,T];L^1(\mathbb{T}^d))}+\|\frac{u^{\varepsilon,\eta}-\bar{u}^{\eta}}{\sqrt{\varepsilon}}-\bar{u}^{1,\eta}\|_{L^1([0,T];L^1(\mathbb{T}^d))}+\|\bar{u}^{1,\eta}-\bar{u}^1\|_{L^1([0,T];L^1(\mathbb{T}^d))},
%\end{align*}
%for
%And now focus on the second term, for fix $\eta>0$, we can choose $\varepsilon>0$ such that
%\begin{align*}
%\|\frac{u^{\varepsilon,\eta}-\bar{u}^{\eta}}{\sqrt{\varepsilon}}-\bar{u}^{1,\eta}\|_{L^1([0,T];L^1(\mathbb{T}^d))}<\frac{\delta}{3}.
%\end{align*}
%So that we have
%\begin{eqnarray}
%\lim_{\varepsilon\rightarrow 0}E\Big\|\frac{u^{\varepsilon}-\bar{u}}{\sqrt{\varepsilon}}-\bar{u}^1\Big\|_{L^1([0,T];L^1(\mathbb{T}^d))}=0.
%\end{eqnarray}
%Then we finished the prove of the central limit theorem of SBE.
\end{proof}

\section{Moderate deviation principle}\label{section4}
In this section, we will prove that
{\color{black}\begin{align*}
X^{\varepsilon}:=\frac{u^{\varepsilon}-\bar{u}}{\sqrt{\varepsilon}\lambda(\varepsilon)}
\end{align*}}
{\color{black}obeys} an LDP on $L^1([0,T];L^1(\mathbb{T}^d))$ with $\lambda(\varepsilon)$ satisfying (\ref{e-43}), which is called moderate deviation principle.

\subsection{A sufficient condition for Large Deviation Principle}
We first introduce some notations and recall a general criteria for  large deviation principle given by \cite{MSZ}. Denote by $\mathcal{E}$ a Polish space with metric $d$, and $\mathcal{B}(\mathcal{E})$ is the Borel $\sigma$-algebra produced by the metric $d$.

\begin{dfn}
(Rate function) A function $I: \mathcal{E}\rightarrow [0,\infty]$ is called a rate function if $I$ is lower semicontinuous.
%A rate function $I$ is called a good rate function if the level set $\{x\in \mathcal{E}: I(x)\leq M\}$ is compact for each $M<\infty$.
\end{dfn}

\begin{dfn}
(Large Deviation Principle) Let $I$ be a rate function on $\mathcal{E}$. A family $\{X^{\varepsilon}\}$ of

\noindent${\color{black}\mathcal{E}}$-valued random elements is said to satisfy the large deviation principle on $\mathcal{E}$ with rate function $I$, if the following two conditions hold.
\begin{description}
  \item[(i)] (Upper bound) For each closed subset $F$ of $\mathcal{E}$,
\begin{equation*}
\limsup_{\varepsilon\rightarrow0}\varepsilon \log P(X^{\varepsilon}\in F)\leq-\inf_{x\in F}I(x).
\end{equation*}
  \item[(ii)] (Lower bound) For each open subset $G$ of $\mathcal{E}$,
\begin{equation*}
\liminf_{\varepsilon\rightarrow0}\varepsilon \log P(X^{\varepsilon}\in G)\geq-\inf_{x\in G}I(x).
\end{equation*}
\end{description}
\end{dfn}
Assume that $W$ is a cylindrical Wiener process on a Hilbert space $U$ defined on $(\Omega, \mathcal{F},\{\mathcal{F}_t\}_{t\in [0,T]}, P)$ (that is, the paths of $W$ take values in $C([0,T];\mathcal{U})$, where $\mathcal{U}$ is another Hilbert space such that the embedding $U\subset \mathcal{U}$ is Hilbert-Schmidt).
The Cameron-Martin space of the Wiener process $\{W(t),t\in[0,T]\}$ is given by
\begin{equation*}
\mathcal{H}_0:=\Big\{h:[0,T]\rightarrow U; h\ {\rm{is\ absolutely\ contimuous\ and\ }} \int_0^T\|\dot{h}(s)\|_{U}^2ds<\infty\Big\}.
\end{equation*}
The space $\mathcal{H}_0$ is a Hilbert space with {\color{black}the} inner product
\begin{align*}
<h_1,h_2>_{\mathcal{H}_0}:=\int_0^T(\dot{h}_1(s),\dot{h}_2(s))_Uds,\ \ {\color{black}{\text{for\ }} h_1,h_2\in\mathcal{H}_0.}
\end{align*}
Denote by  $\mathcal{A}$ the class of $\{\mathcal{F}_t\}-$pblackictable processes $\phi$ belonging to $\mathcal{H}_0$, $P-$a.s.
{\color{black}For $N<\infty$, let
\begin{align*}
S_N:=\Big\{h\in\mathcal{H}_0;\int_0^T\|\dot{h}(s)\|^2_{U}ds\leq N\Big\}.
\end{align*}
}
Here and in the sequel of this paper, we will always refer to the weak topology on the set $S_N$ {\color{black}for which} $S_N$ is a Polish space.
Define
{\color{black}
\begin{align*}
\mathcal{A}_N:=\Big\{\phi\in\mathcal{A};\phi(\omega)\in S_N, {\rm{P-}}a.s.\Big\}.
\end{align*}
}
Recently, a new sufficient condition \big(conditions (a) and (b) in the following Theorem \ref{LDP}\big) {\color{black}implying} the large deviation principle is proposed by Matoussi, Sabagh and Zhang in \cite{MSZ}. It turns out {\color{black}that} this new sufficient condition is suitable for establishing the large deviation principle for  stochastic conservation laws. Combining Budhiraja et al. \cite{BD} and \cite{MSZ}, it follows that
\begin{thm}\label{LDP}
For $\varepsilon>0$, let $\Gamma^{\varepsilon}$ be a measurable mapping from $C([0,T];\mathcal{U}))$ into $\mathcal{E}$. {\color{black}Let

\noindent$X^{\varepsilon}:=\Gamma^{\varepsilon}(W(\cdot))$.} Suppose that $\{\Gamma^{\varepsilon}\}_{\varepsilon>0}$ satisfies the following assumptions: there exists a measurable map $\Gamma^0:C([0,T];\mathcal{U})\rightarrow\mathcal{E}$ such that
\begin{description}
  \item[(a)] for every $N<\infty$ and any family $\{h^{\varepsilon}; \varepsilon>0\}\subset\mathcal{A}_N$, and for any $\delta>0$,
\begin{equation*}
\lim_{\varepsilon\rightarrow0}P(d(Y^{\varepsilon},Z^{\varepsilon})>\delta)=0,
\end{equation*}
where $Y^{\varepsilon}:=\Gamma^{\varepsilon}(W(\cdot)+\frac{1}{\sqrt{\varepsilon}}\int_0^{\cdot}\dot{h}^{\varepsilon}(s)ds), Z^{\varepsilon}:=\Gamma^0(\int_0^{\cdot}\dot{h}^{\varepsilon}(s)ds)$.
  \item[(b)] for every $N<\infty$, the family $\{h_m\}_{m\geq 1}\subset S_N$ that converges to some element $h$ as $m\rightarrow\infty$, $\Gamma^0(\int_0^{\cdot}\dot{h}_m(s)ds)$ converges to $\Gamma^0(\int_0^{\cdot}\dot{h}(s)ds)$ in the space $\mathcal{E}$.
\end{description}
Then the family $\{X^{\varepsilon}\}_{\varepsilon>0}$ satisfies a large deviation principle in $\mathcal{E}$ with the rate function $I$ given by
\begin{equation}\label{rf}
I(g):=\inf_{\big\{h\in\mathcal{H}_0;g=\Gamma^0(\int_0^{\cdot}\dot{h}(s)ds)\big\}}\Big\{\frac{1}{2}\int_0^T\|\dot{h}(s)\|^2_{U}ds\Big\},\ \ g\in\mathcal{E},
\end{equation}
with the convention $\inf\{\emptyset\}=\infty$.
\end{thm}

\subsection{Skeleton equation}
We begin by introducing the map $\Gamma_0$ that will be used to define the rate function and also used to verify conditions (a) and (b) in Theorem \ref{LDP}.

For any $h\in\mathcal{H}_0$, consider the following deterministic equation
\begin{eqnarray}\label{skeleton}
\left\{
  \begin{array}{ll}
   dX_h+ {\rm{div}} (a(\bar{u})X_h(t))dt=\Phi(\bar{u})\dot{h}(t)dt , &  \\
    X_h(x,0)=0, &
  \end{array}
\right.
\end{eqnarray}
 where $\bar{u}\equiv1$.
{\color{black}The equation (\ref{skeleton}) is called the skeleton equation and can be derived briefly as follows.} {\color{black}From (\ref{sbe-1}) and (\ref{r-12}), the process $X^{\varepsilon}=\frac{u^{\varepsilon}-\bar{u}}{\sqrt{\varepsilon}\lambda(\varepsilon)}$ satisfies}
\begin{equation}\label{rrr-15}
\left\{
  \begin{array}{ll}
   dX^{\varepsilon}+ {\rm{div}} \left(\frac{A(u^{\varepsilon})-A(\bar{u})}{\sqrt{\varepsilon}\lambda(\varepsilon)}\right)
   dt=\lambda(\varepsilon)^{-1}\Phi(\bar{u}+\sqrt{\varepsilon}\lambda(\varepsilon)X^{\varepsilon})dW(t),& \\
 X^{\varepsilon}(0)=0. &
  \end{array}
\right.
\end{equation}
Since $\bar{u}=1$, we have $A(u^{\varepsilon})-A(\bar{u})=A(\sqrt{\varepsilon}\lambda(\varepsilon)X^{\varepsilon}+1)-A(1)$.
It implies that $\frac{A(u^{\varepsilon})-A(\bar{u})}{\sqrt{\varepsilon}\lambda(\varepsilon)}$ is a function of $X^{\varepsilon}$.
Define
\begin{align}\label{r-23}
\Psi(\xi):= \frac{A(\sqrt{\varepsilon}\lambda(\varepsilon)\xi+1)-A(1)}{\sqrt{\varepsilon}\lambda(\varepsilon)},
\end{align}
by Hypothesis H, we know that $\Psi\in C^2(\mathbb{R};\mathbb{R}^d)$ and its derivative $\Psi'$ has at most polynomial growth. With the notation of $\Psi$, (\ref{rrr-15}) can be rewritten as \begin{equation}\label{r-21}
\left\{
  \begin{array}{ll}
   dX^{\varepsilon}+ {\rm{div}} \Psi(X^{\varepsilon})
   dt=\lambda(\varepsilon)^{-1}\Phi(\bar{u}+\sqrt{\varepsilon}\lambda(\varepsilon)X^{\varepsilon})dW(t),& \\
 X^{\varepsilon}(0)=0. &
  \end{array}
\right.
\end{equation}
As discussed above, for any $\varepsilon \in (0,1)$, (\ref{r-21}) is also a stochastic conservation law satisfying Hypothesis H.
%That is, there exist constants $C>0, p\geq1$ such that
% have the following expansion
%\begin{align}\notag
%  \frac{A(u^{\varepsilon})-A(\bar{u})}{\sqrt{\varepsilon}\lambda(\varepsilon)}
%  =&\frac{1}{\sqrt{\varepsilon}\lambda(\varepsilon)}\Big[a(1)(\bar{u})^{p-1}(u^{\varepsilon}-\bar{u})
%  +\frac{p(p+1)}{2}(\bar{u})^{p-2}(u^{\varepsilon}-\bar{u})^{2}
%\\ \notag
%  +&\cdots+ \frac{p(p+1)}{2}(\bar{u})^{2}(u^{\varepsilon}-\bar{u})^{p-2}+a(1)\bar{u}(u^{\varepsilon}-\bar{u})^{p-1}+
%  (u^{\varepsilon}-\bar{u})^{p}\Big]\\ \notag
%  =& a(1)X^{\varepsilon}+\frac{p(p+1)}{2}\sqrt{\varepsilon}\lambda(\varepsilon)(X^{\varepsilon})^2+
%  \cdots+\frac{p(p+1)}{2}(\sqrt{\varepsilon}\lambda(\varepsilon))^{p-3}(X^{\varepsilon})^{p-2}\\
%\notag
%  +&
%  a(1)(\sqrt{\varepsilon}\lambda(\varepsilon))^{p-2}(X^{\varepsilon})^{p-1}+(\sqrt{\varepsilon}\lambda(\varepsilon))^{p-1}(X^{\varepsilon})^{p}\\
%  \label{r-10}
%  =: & \Psi(X^{\varepsilon}).
%\end{align}
%Clearly, $\Psi(\xi)$ is polynomial growth with respect to $\xi$.
The well-posedness of (\ref{r-21})  implies that there exists a measurable mapping $\Gamma^{\varepsilon}: C([0,T];\mathcal{U})\rightarrow L^1([0,T];L^1(\mathbb{T}^d))$ such that $\Gamma^{\varepsilon}(W(\cdot)):=X^{\varepsilon}(\cdot)$.

For any $h\in \mathcal{H}_0$, consider the following SPDE
\begin{equation*}
\left\{
  \begin{array}{ll}
    dX_h^{\varepsilon}+ {\rm{div}}  (\Psi(X^{\varepsilon}_h) )dt
=\lambda(\varepsilon)^{-1}\Phi\Big(\bar{u}+\sqrt{\varepsilon}\lambda(\varepsilon)X_h^{\varepsilon}\Big)\Big(dW(t)+\lambda(\varepsilon)\dot{h}(t)dt\Big)
, &  \\
   X_h^{\varepsilon}(0)=0 . &
  \end{array}
\right.
\end{equation*}
Letting $\varepsilon\rightarrow0$, it follows that $X_h^{\varepsilon}$ converges to $X_h$ in some suitable space with $X_h$ satisfying (\ref{skeleton}).
%\begin{equation}\label{skeleton}
%dX_h(t)+\bar{u}\nabla X_h(t)dt=\Phi(\bar{u})\dot{h}(t)dt,
%\end{equation}
%which is the skeleton equation for the moderate deviation of stochastic Burgers equation, with initial value $X_h(0)=0$.
%\subsubsection{Global well-posedness of the skeleton equation}\label{s-1}

Regarding to (\ref{skeleton}), it is a special case of the skeleton equation (4.3) in \cite{DWZZ} with $A(\xi)=a(1)\xi$ and $\Phi(\xi)=\Phi(1)$.
Now, we introduce the definition of kinetic solution of (\ref{skeleton}) from \cite{DWZZ}.
\begin{dfn}(Kinetic solution)
A measurable function $X_h: \mathbb{T}^d\times [0,T]\rightarrow \mathbb{R}$ is said to be a kinetic solution to (\ref{skeleton}), if for any $p\geq 1$, there exists $C_p\geq 0$ such that
\begin{equation*}
\underset{0\leq t\leq T}{{\rm{ess\sup}}}\ \|X_h(t)\|^p_{L^p(\mathbb{T}^d)}\leq C_p,
\end{equation*}
and if there exists a measure $m_h\in \mathcal{M}^+_0(\mathbb{T}^d\times [0,T]\times \mathbb{R})$ such that $f_h:= I_{X_h>\xi}$ satisfies that for all $\varphi\in C^1_c(\mathbb{T}^d\times [0,T)\times \mathbb{R})$,
\begin{align*}\notag
&\int^T_0\langle f_h(t), \partial_t \varphi(t)\rangle dt+\langle f_0, \varphi(0)\rangle +\int^T_0\langle f_h(t), a(1)\cdot \nabla  \varphi (t)\rangle dt\\
&=-\sum_{k\geq 1}\int^T_0\int_{\mathbb{T}^d} g_k(x,\bar{u})\varphi (x,t, X_h(x,t))\dot{h}^k(t)dxdt + m_h(\partial_{\xi} \varphi),
\end{align*}
where $f_0(x,\xi)=I_{X_0>\xi}=I_{0>\xi}$.
%, $\mathcal{M}^+_0(\mathbb{T}^d\times [0,T]\times \mathbb{R})$ is defined in below of Definition \ref{dfn-3}.
\end{dfn}

%The above brackets $\langle\cdot,\cdot\rangle$ to denote the duality between $C^{\infty}_c(\mathbb{T}^d\times \mathbb{R})$ and the space of distributions over $\mathbb{T}^d\times \mathbb{R}$. In what follows, we will denote similarly the integral
%\[
%\langle F, G \rangle=\int_{\mathbb{T}^d}\int_{\mathbb{R}}F(x,\xi)G(x,\xi)dxd\xi, \quad F\in L^p(\mathbb{T}^d\times \mathbb{R}), G\in L^q(\mathbb{T}^d\times \mathbb{R}),
%\]
%where $1\leq p\leq \infty$ and $q$ is the conjugate exponent of $p$. We have used the shorthand
%\[
%m(\phi)=\int_{\mathbb{T}^d\times[0,T]\times \mathbb{R}}\phi(x,t,\xi)dm(x,t,\xi), \quad  \phi\in C_b(\mathbb{T}^d\times[0,T]\times \mathbb{R}).
%\]

Referring to \cite{DWZZ}, we have the following well-posedness result for (\ref{skeleton}).
\begin{thm}\label{thm-10}
(Well-posedness)
Under Hypothesis H, for any $T>0$, (\ref{skeleton}) admits a unique kinetic solution $X_h$ on $[0,T]$.
\end{thm}
In view of Theorem \ref{thm-10}, we can define a mapping $\Gamma^0: C([0,T];\mathcal{U})\rightarrow L^1([0,T];L^1(\mathbb{T}^d))$ by
\begin{eqnarray*}
\Gamma^0(\check{h}):=\left\{
                   \begin{array}{ll}
                      X_{h}, & {\rm{if}}\ \check{h}= \int^{\cdot}_0 \dot{h}(s)ds, \ {\rm{for\ some}}\ h\in \mathcal{H}_0,\\
                    0, & {\rm{otherwise}},
                   \end{array}
                  \right.
\end{eqnarray*}
where $X_h$ is the solution of equation (\ref{skeleton}).

Moreover, by Theorem 5.6 in \cite{DWZZ}, we have the continuity of the mapping $\Gamma^0$.
\begin{thm}\label{cse}
Fix $N>0$. Assume $\{h_m\}_{m\geq1}\subset S_N$ that converges to some element $h$ as $m\rightarrow\infty$, then $X_m$ converges to $X_h$ in $L^1([0,T]; L^1(\mathbb{T}^d))$, where $X_m$ is the kinetic solution to the skeleton equation (\ref{skeleton}) with $h$ replaced by $h_m$.
\end{thm}

\subsection{Proof of moderate deviation principle}\label{MDPsec}
In this section, we focus on the proof of the main result. It reads as follows.
\begin{thm}\label{MDP} Assume Hypothesis H is in force.
For the kinetic solution $u^{\varepsilon}$ of (\ref{sbe-1}), $\frac{u^{\varepsilon}-\bar{u}}{\sqrt{\varepsilon}\lambda(\varepsilon)}$ satisfies LDP on $L^1([0,T];L^1(\mathbb{T}^d))$ with speed $\lambda^2(\varepsilon)$ and with rate function $I$ defined by (\ref{rf}), that is
\begin{description}
  \item[(I)] for any closed subset $F$ of $L^1([0,T];L^1(\mathbb{T}^d))$,
\begin{equation*}
\limsup_{\varepsilon\rightarrow0}\frac{1}{\lambda(\varepsilon)^2}\log P\Big(\frac{u^{\varepsilon}-\bar{u}}{\sqrt{\varepsilon}\lambda(\varepsilon)}\in F\Big)\leq-\inf_{x\in F}I(x);
\end{equation*}
  \item[(II)] for each open subset $G$ of $L^1([0,T];L^1(\mathbb{T}^d))$,
\begin{equation*}
\liminf_{\varepsilon\rightarrow0}\frac{1}{\lambda(\varepsilon)^2}\log P\Big(\frac{u^{\varepsilon}-\bar{u}}{\sqrt{\varepsilon}\lambda(\varepsilon)}\in G\Big)\geq-\inf_{x\in G}I(x).
\end{equation*}
\end{description}
\end{thm}
According to Theorem \ref{LDP}, in order to establish Theorem \ref{MDP}, we only need to verify sufficient conditions (a) and (b). Clearly, the condition (b) has been proved by Theorem \ref{cse}. Therefore, it remains to prove condition (a).

%Recall that $X^{\varepsilon}=\frac{u^{\varepsilon}-\bar{u}}{\sqrt{\varepsilon}\lambda(\varepsilon)}$ satisfies
%\begin{eqnarray*}
%dX^{\varepsilon}+X^{\varepsilon}\nabla (\sqrt{\varepsilon}\lambda(\varepsilon)X^{\varepsilon})dt+\bar{u}\nabla X^{\varepsilon}dt=\lambda(\varepsilon)^{-1}\Phi(1+\sqrt{\varepsilon}\lambda(\varepsilon)X^{\varepsilon})dW(t),
%\end{eqnarray*}
%with $X^{\varepsilon}(0)=0$.
For any $\{h^{\varepsilon}\}_{\varepsilon>0}\subset\mathcal{A}_N$, we consider
\begin{eqnarray}\label{rrr-6}
\left\{
  \begin{array}{ll}
  d\bar{X}^{\varepsilon}+ {\rm{div}} (\Psi(\bar{X}^{\varepsilon}))dt
=\lambda^{-1}(\varepsilon)\Phi\Big(1+\sqrt{\varepsilon}\lambda(\varepsilon)\bar{X}^{\varepsilon}\Big)dW(t)+\Phi(1+\sqrt{\varepsilon}\lambda(\varepsilon)\bar{X}^{\varepsilon})\dot{h}^{\varepsilon}(t)d t,\\
\bar{X}^{\varepsilon}(0)=0,
  \end{array}
\right.
\end{eqnarray}
where $\Psi$ is defined by (\ref{r-23}).
%\begin{eqnarray}\notag
%d\bar{X}^{\varepsilon}+\bar{X}^{\varepsilon}\nabla (\sqrt{\varepsilon}\lambda(\varepsilon)\bar{X}^{\varepsilon})dt
%+\nabla \bar{X}^{\varepsilon}dt
%&=&\lambda(\varepsilon)^{-1}\Phi(1+\sqrt{\varepsilon}\lambda(\varepsilon)\bar{X}^{\varepsilon})dW(t)\\
%\label{rrr-6}
%&&+\Phi(1+\sqrt{\varepsilon}\lambda(\varepsilon)\bar{X}^{\varepsilon})\dot{h}^{\varepsilon}(t)d t,
%\end{eqnarray}
%with initial value $\bar{X}^{\varepsilon}(0)=0$.
%Denote by $A_{\varepsilon}(\xi):=\sqrt{\varepsilon}\lambda(\varepsilon)\frac{\xi^2}{2}+\xi$, then $a_{\varepsilon}(\xi):=A_{\varepsilon}'(\xi)=\sqrt{\varepsilon}\lambda(\varepsilon)\xi+1$.
Clearly, (\ref{rrr-6}) is a special case of the stochastic controlled equation (6.1) in \cite{DWZZ}, then we deduce that there exists a unique kinetic solution $\bar{X}^{\varepsilon}$ satisfying that for any $p\geq1$
\begin{align}\label{rrr-14}
\sup_{\varepsilon\in (0,1)}E\Big(\underset{0\leq t\leq T}{{\rm{ess\sup}}}\ \|\bar{X}^{\varepsilon}(t)\|_{L^p(\mathbb{T}^d)}^p\Big)\leq C_p,
\end{align}
and there exists a kinetic measure  $\bar{m}^\varepsilon\in\mathcal{M}^+_0(\mathbb{T}^d\times [0,T]\times \mathbb{R})$  such that $f^{\varepsilon}:=I_{\bar{X}^{\varepsilon}>\xi}$ satisfies for any $\varphi\in C_c^1(\mathbb{T}^d\times[0,T)\times\mathbb{R})$,
\begin{align*}
&\int_0^T\langle f^{\varepsilon}(t),\partial_t\varphi(t)\rangle dt+<f_0,\varphi(0)>
+\int_0^T\langle f^{\varepsilon}(t),\Psi'(\xi)\cdot \nabla \varphi(t)\rangle dt\\
=&-\lambda^{-1}(\varepsilon)\sum_{k\geq1}\int_0^T\int_{\mathbb{T}^d}g_k(x,1+\sqrt{\varepsilon}\lambda(\varepsilon)\bar{X}^{\varepsilon})\varphi(x,t,\bar{X}^{\varepsilon})dxd\beta_k(t)\\
&-\frac{1}{2\lambda^2(\varepsilon)}\int_0^T\int_{\mathbb{T}^d}\partial_{\xi}\varphi(x,t,\bar{X}^{\varepsilon})G^2(x,1+\sqrt{\varepsilon}\lambda(\varepsilon)\bar{X}^{\varepsilon})dxdt\\
&-\sum_{k\geq1}\int_0^T\int_{\mathbb{T}^d}\varphi(x,t,\bar{X}^{\varepsilon})g_k(x,1+\sqrt{\varepsilon}\lambda(\varepsilon)\bar{X}^{\varepsilon})\dot{h}^{\varepsilon,k}(t)dxdt+\bar{m}^{\varepsilon}(\partial_{\xi}\varphi),\ \ a.s.,
\end{align*}
where $\{\dot{h}^{\varepsilon,k}\}_{k\geq1}$ are the Fourier coefficients of $\dot{h}^{\varepsilon}$, that is $\dot{h}^{\varepsilon}(t)=\sum_{k\geq1}\dot{h}^{\varepsilon,k}(t)e_k$.
According to the definition of
$\Gamma^{\varepsilon}$, it follows that $\Gamma^{\varepsilon}(W(\cdot)+\lambda(\varepsilon)\int_0^{\cdot}\dot{h}^{\varepsilon}(s)ds)=\bar{X}^{\varepsilon}(\cdot)$.

%Referring to the definition of $\Gamma^{\varepsilon}$, it is clear  that $\Gamma^{\varepsilon}(W(\cdot)+\lambda(\varepsilon)\int_0^{\cdot}\dot{h}^{\varepsilon}(s)ds)=\bar{X}^{\varepsilon}(\cdot)$.

Now, we are in a position to verify the condition (a) in Theorem \ref{LDP}.
\begin{thm}
For every $N<\infty$, let $\{h^{\varepsilon}\}_{\varepsilon>0}\subset\mathcal{A}_N$. Then
\begin{equation}\label{r-11}
\Big\|\Gamma^{\varepsilon}\left(W(\cdot)+\lambda(\varepsilon)\int_0^{\cdot}\dot{h}^{\varepsilon}(s)ds\right)-\Gamma^0\left(\int_0^{\cdot}\dot{h}^{\varepsilon}(s)ds\right)\Big\|_{L^1([0,T];L^1(\mathbb{T}^d))}\rightarrow0,
\end{equation}
in probability, as $\varepsilon\rightarrow0$.
\end{thm}
\begin{proof}
Recall that $\bar{X}^{\varepsilon}=\Gamma^{\varepsilon}\Big(W(\cdot)+\lambda(\varepsilon)\int_0^{\cdot}\dot{h}^{\varepsilon}(s)ds\Big)$ is the kinetic solution to (\ref{rrr-6}) and $Y^{\varepsilon}(\cdot):=\Gamma^0\Big(\int_0^{\cdot}\dot{h}^{\varepsilon}(s)ds\Big)$ is the kinetic solution to the skeleton equation (\ref{skeleton}) with $h$ replaced by $h^{\varepsilon}$.
As discussed in the section of the central limit theorem, (\ref{r-11}) cannot be established  by applying the doubling variables method directly to $\bar{X}^{\varepsilon}$ and $Y^{\varepsilon}$ due to the lack of symmetry. To solve it, we adopt the same method as the proof of CLT to introduce some auxiliary parabolic approximation processes, which are symmetric with the original equations. Concretely, for $\eta>0$, let $\bar{X}^{\varepsilon,\eta}$ and $Y^{\varepsilon,\eta}$ be defined by
\begin{eqnarray*}
\left\{
  \begin{array}{ll}
  d\bar{X}^{\varepsilon,\eta}+ {\rm{div}} (\Psi(\bar{X}^{\varepsilon,\eta}))dt
=\eta\Delta \bar{X}^{\varepsilon,\eta}dt+ \lambda^{-1}(\varepsilon)\Phi\Big(1+\sqrt{\varepsilon}\lambda(\varepsilon)\bar{X}^{\varepsilon,\eta}\Big)dW(t)
+\Phi(1+\sqrt{\varepsilon}\lambda(\varepsilon)\bar{X}^{\varepsilon,\eta})\dot{h}^{\varepsilon}(t)d t,\\
\bar{X}^{\varepsilon,\eta}(0)=0,
  \end{array}
\right.
\end{eqnarray*}
and
\begin{eqnarray*}
\left\{
  \begin{array}{ll}
   dY^{\varepsilon,\eta}+ {\rm{div}} (a(\bar{u})Y^{\varepsilon,\eta}(t))dt=\eta\Delta Y^{\varepsilon,\eta}dt+\Phi(\bar{u})\dot{h}^{\varepsilon}(t)dt , &  \\
    Y^{\varepsilon,\eta}(x,0)=0. &
  \end{array}
\right.
\end{eqnarray*}
With the notations of $\bar{X}^{\varepsilon,\eta}$ and $Y^{\varepsilon,\eta}$, it follows that
\begin{align}\notag
&E\|\bar{X}^{\varepsilon}-Y^{\varepsilon}\|_{L^1([0,T];L^1(\mathbb{T}^d))}\\ \notag
\leq &E\|X^{\varepsilon}-X^{\varepsilon,\eta}\|_{L^1([0,T];L^1(\mathbb{T}^d))}
+E\Big\|X^{\varepsilon,\eta}-Y^{\varepsilon,\eta}\Big\|_{L^1([0,T];L^1(\mathbb{T}^d))}
+E\|Y^{\varepsilon,\eta}-Y^{\varepsilon}\|_{L^1([0,T];L^1(\mathbb{T}^d))}.
\end{align}
Referring to Proposition 5.5 in \cite{DWZZ}, it gives
\begin{align*}
\lim_{\eta\rightarrow 0}\sup_{\varepsilon\in (0,1)}  E\|Y^{\varepsilon,\eta}-Y^{\varepsilon}\|_{L^1([0,T];L^1(\mathbb{T}^d))}=0.
\end{align*}
The remaining two terms can be treated similarly to the proof of CLT, we omit it.

\end{proof}

\noindent{\bf  Acknowledgements}\quad  This work is partly supported by Beijing Natural Science Foundation (No. 1212008), National Natural Science Foundation of China (No. 12171032,11971227,12071123), Key Laboratory of Random Complex Structures and Data Science, Academy of Mathematics and Systems Science, Chinese Academy of Sciences (No. 2008DP173182), Beijing Institute of Technology Research Fund Program for Young Scholars and MIIT Key Laboratory of Mathematical Theory and Computation in Information Security.

%%%%%%%%%%%%%%%%%%%%%%%%%%%%%%%%%%%%%%%%%%%%%%%%%%%%%%
\def\refname{ References}

\end{document}